\documentclass[12pt]{amsart}
\usepackage{a4wide,enumerate,xcolor,graphicx}
\usepackage{amsmath,comment}
\usepackage{scrextend} 
\usepackage{array,multirow,makecell}
\usepackage[scale = .8]{geometry}
\usepackage[normalem]{ulem}
\usepackage{soul}
\usepackage{mathrsfs}
\usepackage{dsfont}
\usepackage{pstricks}
\usepackage{mathtools}
\usepackage{float}
\usepackage{xparse}
\usepackage{xspace}
\usepackage{xcolor}
\usepackage{hyperref}
\usepackage{tikz, pgfplots}

\let\eps\varepsilon
\newcommand{\N}{{\mathbb N}}
\newcommand{\R}{{\mathbb R}}
\newcommand{\dx}{\Delta x}
\newcommand{\dd}{{\mathrm d}}

\newcommand{\Z}{{\mathbb Z}}
\newcommand{\M}{{\mathbb M}}
\newcommand{\D}{{\mathcal D}}

\newcommand{\I}{{\mathcal I}}
\newcommand{\Tor}{{\mathbb T}}

\newcommand{\T}{{\mathcal T}}

\newcommand{\F}{\mathcal{F}}
\newcommand{\lla}{\left\langle}
\newcommand{\rra}{\right\rangle}

\newcolumntype{C}[1]{>{\centering\arraybackslash }b{#1}}

\newtheorem{theorem}{Theorem}
\newtheorem{lemma}[theorem]{Lemma}

\newtheorem{proposition}[theorem]{Proposition}
\newtheorem{remark}[theorem]{Remark}
\newtheorem{corollary}[theorem]{Corollary}
\newtheorem{definition}{Definition}

\begin{document}

\title[Finite volume scheme for nonlocal cross-diffusion]{Study of an entropy dissipating finite volume scheme for a nonlocal cross-diffusion system}

\author[M. Herda]{Maxime Herda}
\address{Inria, Univ. Lille, CNRS, UMR 8524--Laboratoire Paul Painlev\'e, 59000 Lille}
\email{maxime.herda@inria.fr}

\author[A. Zurek]{Antoine Zurek}
\address{Universit\'e de Technologie de Compi\`egne, LMAC, 60200 Compiègne, France}
\email{antoine.zurek@utc.fr}

\date{\today}

\begin{abstract}
In this paper we analyse a finite volume scheme for a nonlocal version of the Shigesada-Kawazaki-Teramoto (SKT) cross-diffusion system. We prove the existence of solutions to the scheme, derive qualitative properties of the solutions and prove its convergence. The proofs rely on a discrete entropy-dissipation inequality, discrete compactness arguments, and on the novel adaptation of the so-called duality method at the discrete level. Finally, thanks to numerical experiments, we investigate the influence of the nonlocality in the system: on convergence properties of the scheme, as an approximation of the local system and on the development of diffusive instabilities.

\smallskip


\smallskip

\end{abstract}

\maketitle


\section{Introduction}

We are interested in the numerical discretization of the following nonlocal cross-diffusion system
\begin{align}\label{4.SKT1}
\partial_t u_1 - \Delta((d_1 + d_{11} \, \sigma_1 \ast u_1 + d_{12} \, \rho_1 \ast u_2 ) u_1) &= R_1(u_1,u_2),\\\label{4.SKT2}
\partial_t u_2 - \Delta((d_2 + d_{21} \, \rho_2 \ast u_1 + d_{22} \sigma_2 \ast u_2 ) u_2) &= R_2(u_1,u_2),
\end{align}
on a periodic domain $\Omega= \Tor^d$ ($d \leq 3$). For a given final time $T$ we denote the space time domain by  $Q_T = \Omega \times (0,T)$. The parameters $d_1$, $d_2$, $d_{11}$, $d_{12}$, $d_{21}$ and $d_{22}$ are some positive constants and $\rho_1$,  $\rho_2$, $\sigma_1$ and $\sigma_2$ are non-negative convolution kernels. System \eqref{4.SKT1}--\eqref{4.SKT2} is supplemented with initial conditions
\begin{align}\label{4.ic}
u_1(\cdot,0) = u_1^0(\cdot), \quad u_2(\cdot,0) = u_2^0(\cdot).
\end{align}
In the case where the convolution kernels are given by the Dirac measure $\rho_1 = \rho_2 = \sigma_1 = \sigma_2 = \delta_0$, the system coincides with the celebrated Shigesada, Kawasaki, and Teramoto (SKT) population model \cite{SKT79} which can describe segregation phenomena between competing species. It writes
\begin{align}\label{4.local.SKT1}
\partial_t u_1 - \Delta((d_1 + d_{11} \, u_1 + d_{12} \, u_2 ) u_1) &= R_1(u_1,u_2),\\\label{4.local.SKT2}
\partial_t u_2 - \Delta((d_2 + d_{21} \, u_1 + d_{22} \, u_2 ) u_2) &= R_2(u_1,u_2).
\end{align}

Nonlocal cross-diffusion systems appear naturally as a mean field type of limit of interacting many-particle systems. For instance, the model \eqref{4.SKT1}--\eqref{4.SKT2} was introduced in \cite{FM15} as the large population limit of a stochastic individual model. If these particle systems allow a precise description of the interactions between individuals, their numerical approximations are very time-consuming. Then, it is reasonable to investigate simpler macroscopic models. In this context we see nonlocal cross-diffusion models as intermediate models between individual based models and local cross-diffusion models. This interpretation has been mathematically justified in the literature, see \cite{CDHJ21,DM21,JPZ22,AM20}, where the derivation of some local cross-diffusion models from nonlocal models (some of them derived from microscopic models) are shown.

Besides, nonlocal cross-diffusion models can be more than a mathematical intermediate between two scales. Indeed, in population dynamics, they can model nonlocal sensing, as diffusion of a species is impacted by the population located (respectively to their position) on the support of the convolution kernels, see \cite{GHLP21,PL19}. In the model \eqref{4.SKT1}--\eqref{4.SKT2} assume for instance that $\rho_1$ is supported away from $0$. Then the resulting effect of the nonlocal cross-diffusion term is to enhance the diffusion of species $1$ when species $2$ is away, modeling for instance a hunting behavior in a predator-prey model. This could hardly be reproduced by local cross-diffusion terms. 

The ability of the nonlocal cross-diffusion terms to model the dynamics of some natural phenomena explain the use of such models in other contexts. They are for instance applied to describe cell sorting \cite{MT15,PBSG15}, tumour growth \cite{dtGC14}, opinion formation \cite{DMPW09} or interactions between spiking neurons \cite{BFFT12} (just to name a few). In particular, the development of reliable numerical methods to approximate the solutions of nonlocal cross-diffusion systems can enhance our understanding of the ``physical'' mechanisms described by them. As a by-product this could also help the development of efficient models describing complex phenomena.

Motivated by these reasons, this manuscript deals with the design and analysis of a robust numerical scheme for \eqref{4.SKT1}--\eqref{4.ic}. Our approach is inspired by the analysis performed at the continuous level in \cite{DM21,AM20}. In particular, in \cite{DM21} the authors show that there is a persisting entropy structure in the nonlocal case which yields a crucial \emph{a priori} estimate for the analysis of the model. This extends for instance the approach developed in \cite{Jue15,Jue16} in the local case. Indeed, it was shown that for the system  \eqref{4.SKT1}--\eqref{4.ic} without reaction terms $R_1 = R_2 = 0$ and under the following symmetry hypotheses on the convolution kernels
\begin{equation}\label{eq:hypoth}
\left\{
 \begin{aligned}
    &\rho_1(x) = \rho_2(-x)= \rho(x), \\
    &\sigma_1(x) = \sigma_1(-x),\\
    &\sigma_2(x) = \sigma_2(-x),
 \end{aligned}\right.
\end{equation}
that the following entropy functional
\begin{align*}
H(u_1,u_2) = \int_{\Omega} \frac{1}{d_{12}} \left[ u_1(\log(u_1)-1)+1 \right] \, \dd x + \int_{\Omega} \frac{1}{d_{21}} \left[ u_2(\log(u_2)-1)+1 \right] \, \dd x\,,
\end{align*}
is dissipated along solutions of \eqref{4.SKT1}--\eqref{4.SKT2}. More precisely one has
\begin{multline}\label{4.einonlocal.cont}
\frac{\mathrm{d}}{\mathrm{d}t}H(u_1,u_2) + \frac{2 d_{11}}{d_{12}} \int_{\Omega} \int_{\Omega} \sigma_1(y) \left(\sqrt{u_1(x-y)} \, \nabla \sqrt{u_1(x)} + \sqrt{u_1(x)} \nabla \sqrt{u_1(x-y)} \right)^2\, \mathrm{d}x\mathrm{d}y\\
+ \frac{2 d_{22}}{d_{21}} \int_{\Omega}\int_{\Omega} \sigma_2(y) \left(\sqrt{u_2(x-y)} \, \nabla \sqrt{u_2(x)} + \sqrt{u_2(x)} \nabla \sqrt{u_2(x-y)} \right)^2\, \mathrm{d}x\mathrm{d}y \\
+4 \frac{d_1}{d_{12}} \int_{\Omega} |\nabla \sqrt{u_1}|^2 \, \mathrm{d}x 
+ 4\frac{d_2}{d_{21}} \int_{\Omega} |\nabla \sqrt{u_2}|^2 +4 \int_{\Omega} \rho(y) \int_{\Omega} |\nabla \sqrt{u_1(x) u_2(x-y)}|^2 \, \mathrm{d}x \mathrm{d}y = 0\,.
\end{multline}
Observe that if the convolution kernels are given by the Dirac measure $\rho_1 = \rho_2 = \sigma_1 = \sigma_2 = \delta_0$, then 
\begin{multline*}
\frac{\mathrm{d}}{\mathrm{d}t}H(u_1,u_2) +2 \frac{d_{11}}{d_{12}} \int_{\Omega} |\nabla u_1|^2 \, \mathrm{d}x + 2 \frac{d_{22}}{d_{21}} \int_{\Omega} |\nabla u_2|^2 \, \mathrm{d}x \\+ 4 \frac{d_1}{d_{12}} \int_{\Omega} |\nabla \sqrt{u_1}|^2 \, \mathrm{d}x + 4\frac{d_2}{d_{21}} \int_{\Omega} |\nabla \sqrt{u_2}|^2 + 4 \int_{\Omega} |\nabla \sqrt{u_1 \, u_2}|^2 \, \mathrm{d}x = 0\,,
\end{multline*}
which was already known for the local SKT system (see \cite{CJ03,CJ06,GGJ03}).

The fact that \eqref{4.SKT1}--\eqref{4.SKT2} admits a Lyapunov functional is crucial for the study of the system. Indeed, in \cite{DM21}, the authors used \eqref{4.einonlocal.cont} together with the so-called duality method, see \cite{DLMT15,LM17,AM20}, in order to prove (assuming \eqref{eq:hypoth} and without reaction terms) the existence of distributional solutions to \eqref{4.SKT1}--\eqref{4.ic}.

\begin{definition}\label{def.distrib.solution}
Given $T> 0$, let $\rho_1$, $\rho_2$, $\sigma_1$ and $\sigma_2$ be some functions in $L^\infty(\Omega)$ and $u^0_1$ and $u^0_2$ be some initial functions in $L^1(\Omega)$. Then, we say that the measurable functions $u_1,u_2 : Q_T \to \R_+$ are distributional solutions to \eqref{4.SKT1}--\eqref{4.ic} if for every $\phi \in C^\infty_0(\Omega \times [0,T))$ it holds
\begin{align}\label{weak1}
 \int_{Q_T}\Big( u_1 \partial_t \phi +\left(d_1 u_1 + d_{11} \sigma_1 \ast u_1  u_1 + d_{12} \rho_1 \ast u_2  u_1 \right) \Delta \phi \Big) \mathrm{d}x\mathrm{d}t =- \int_{\Tor^d} u^0_1(x)  \phi(x,0) \mathrm{d}x,
\end{align}
and
\begin{align}\label{weak2}
 \int_{Q_T}\Big( u_2 \partial_t \phi +\left(d_2 u_2 + d_{21} \rho_2 \ast u_1  u_2 + d_{22} \sigma_2 \ast u_2  u_2 \right) \Delta \phi \Big) \mathrm{d}x\mathrm{d}t =- \int_{\Tor^d} u^0_2(x)  \phi(x,0) \mathrm{d}x.
\end{align}
\end{definition}

In this paper we propose and analyze a finite volume scheme for \eqref{4.SKT1}--\eqref{4.ic}. A particular focus is put on 
\begin{itemize}
 \item[(i)] the preservation of the entropy  dissipation property at the discrete level;
 \item[(ii)] the non-negativity of the solution;
 \item[(iii)] the possibility to use the scheme in both the nonlocal and local regimes.
\end{itemize}

In order to achieve these goals we will design a fully implicit two point flux approximation (TPFA) finite volume scheme. As in the study of some numerical schemes for local cross-diffusion systems, see for instance the following (non-exhaustive) list of contributions  \cite{ABR11,BPS22,CCGJ19,CB20,JZ20,SCS19}, the preservation of the entropy dissipation property at the discrete level is crucial. This ensures well-posedness and global stability in time \cite{ChHe20,FiHe17} as well as with respect to the choice of convolution kernels (see Theorem~\ref{thm.SKTexi}). Some of these methods are reminiscent of the second author's work in \cite{JZ20} concerning the study of a finite volume scheme for the local SKT system. Besides, we are able to obtain additional estimates on the solution (see Theorem~\ref{thm.prop}) by adapting the duality method (see \cite{DLMT15,LM17,AM20}) at the discrete level. This technique relies on the study of a discretized Kolmogorov equation, see Section \ref{sec.Kolmo}. The convergence of solutions of the numerical scheme towards distributional solutions in the sense of Definition \ref{def.distrib.solution} is shown in Theorem~\ref{thm.convSKT}. Let us mention that only $L^\infty$ regularity of the convolution kernels is required to obtain convergence of the scheme. With smoother kernels one additionally obtains local in time $L^\infty$ bounds on the discrete solution that are uniform in the mesh size (see Theorem~\ref{thm.prop}).

Let us notice that there already exists some works dealing with the design and the analysis of finite volume numerical schemes for nonlocal cross-diffusion systems. Indeed, in \cite{ABLS15,ABS15} the convergence of some semi-implicit TPFA finite volume schemes are proved. The convergence proofs are based (as in this work) on the adaptation at the discrete level of a Kruzhkov's compactness result \cite{Kru69} (see also \cite{Re30}) obtained in \cite{ABR11}. We mention \cite{BCPS20} where numerical experiments are shown to illustrate the formation of gaps for a class of nonlocal cross-diffusion systems. In this paper the authors applied an explicit in time finite
volume scheme first introduced in \cite{CCH15} and then extended for the multi-species case in \cite{CHS18}. This scheme is a positivity and entropy preserving method as shown in \cite{CHS18}. Finally, we also refer to \cite{CFS20}. In this contribution the convergence of a semi-discrete finite volume scheme is proved. This scheme is also positivity preserving which allows the authors to establish a discrete energy estimate.


In order to illustrate and complement the theoretical results, we present several numerical experiments in the last section of this paper. We compute the experimental order of convergence of the numerical method when the  mesh size goes to $0$ for various initial data and convolution kernels. Then, for a fixed mesh, we investigate the rate of convergence for different metrics of the discrete solutions of the nonlocal system towards solution of the local system when convolution kernels tends to Dirac measures. Finally, we perform simulations of the model with nonzero reaction terms with parameters chosen to describe a prey-predator system with either linear diffusion or non-local cross diffusion modelling hunting behavior. For these models we illustrate the persistence and the modification of Turing patterns in the presence of cross-diffusion. 

The paper is organized as follows, in Section \ref{sec.main} we introduce the scheme and state our main results. Section \ref{subsec.proofexiSKT} is concerned with the proof of existence of positive solutions to the scheme. We introduce the discrete Kolmogorov equation in Section \ref{sec.Kolmo}. Then we deduce from the study of this problem some qualitative properties satisfied by the solutions to our scheme in Section \ref{sec.proof.thm.prop}. Sections \ref{sec.conv} deals with the convergence of the scheme. Finally, in Section \ref{sec.num}, we discuss the implementation and show some numerical experiments in one and two space dimensions.

\section{Numerical scheme and main results}\label{sec.main}

The results of this paper apply to a periodic domain $\Omega = \prod_{i=1}^d \mathbb{R}/L_i\mathbb{Z}$, where $L_1,\dots,L_d>0$. However, for the sake of readability we will assume from now on that $d=1$ and $L_1 = 1$, namely $\Omega = \Tor = \mathbb{R}/\mathbb{Z}$. The generalization in higher dimensions on Cartesian grid is immediate by defining the scheme as the tensorization of the one dimensional scheme.

\subsection{Notations and definitions}

Let us define $N \geq 1$ and $\Delta x = 1/N$. A uniform mesh $\mathcal{T}$ of $\Tor$ consists in a finite sequence of cells denoted by
\[
K_i = (x_{i-\frac{1}{2}}, x_{i+\frac{1}{2}})\,, \quad i \in \I = \Z/N\Z.
\]
centered at $x_i=i\Delta x$ and with extremities $x_{i\pm\frac{1}{2}} = (i\pm\frac12)\Delta x$. For $T>0$ given, we define an integer $N_T$ and a time step $\Delta t= T/N_T$ and we introduce the sequence $(t_k)_{0\leq k \leq N_T}$ with $t_k = k \Delta t$. We denote by $\D$ a space-time discretization of $Q_T= \Tor \times (0,T)$ composed of a space discretization $\T$ of $\Tor$ and the values $(\Delta t, N_T)$.

Let us now introduce some discrete norms on the space of piecewise constant functions in space
\begin{align*}
\mathcal{H}_{\T} = \Bigg\lbrace w : \Tor  \to \R \, \, : \, \, w(x)= \sum_{i \in \I} w_i \mathbf{1}_{K_i}(x)  \Bigg\rbrace.
\end{align*}
For $p \in [1,\infty)$, we define the discrete $W^{1,p}$ seminorm and discrete $W^{1,p}$ norm on $\mathcal{H}_\T$ by
\begin{align*}
|w|_{1,p,\T} = \left(\sum_{i \in \I} \Delta x \left|\frac{w_{i+1}-w_i}{\Delta x}\right|^p\right)^{\frac{1}{p}}, \quad \|w\|_{1,p,\T} = |w|_{1,p,\T} + \|w\|_{L^p(\Tor)},
\end{align*}
where, for $p \in [1, \infty)$, the norm $\|\cdot\|_{L^p(\Tor)}$ denotes the usual $L^p(\Tor)$ norm. In the case $p=\infty$, we denote $\|\cdot\|_{L^\infty(\Tor)}$ the $L^\infty(\Tor)$ norm given by
\begin{align*}
\|w\|_{L^\infty(\Tor)} = \max_{i \in \I} |w_i|, \quad \forall w \in \mathcal{H}_\T.
\end{align*}
Let us also recall the definition of the space $BV(\Tor)$, see \cite{AFP00} for more details. A function $w \in L^1(\Tor)$ belongs to the space $BV(\Tor)$ if its total variation $TV(w)$ given by
\begin{align*}
TV(w) = \sup\left\lbrace \int_\Tor w(x) \, \partial_x \phi(x) \, \mathrm{d}x, \quad \phi \in C^1_c(\Tor), \quad |\phi(x)| \leq 1 \quad \forall x \in \Tor \right\rbrace,
\end{align*}
is finite. We endow the space $BV(\Tor)$ with the norm 
$$
	\|w\|_{BV(\Tor)} = \|w\|_{L^1(\Tor)} + TV(w), \quad \forall w \in BV(\Tor).
$$
In particular, we notice that for each function $w \in \mathcal{H}_\T \cap BV(\Tor)$ we have $\|w\|_{BV(\Tor)} = \|w\|_{1,1,\T}$.

Finally we introduce the space $\mathcal{H}_\D$ of piecewise constant in time functions with values in $\mathcal{H}_\T$,
\begin{align*}
\mathcal{H}_{\D} = \Bigg\lbrace w : \Tor \times [0,T]  \to \R \, \, : \, \, w(x,t)= \sum_{k=1}^{N_T} w^k(x) \mathbf{1}_{(t_{k-1},t_k]}(t)  \Bigg\rbrace.
\end{align*}
This space can be equipped, for $(p,q) \in [1,\infty)^2$, with the following discrete $L^q(0,T;W^{1,p}(\Tor))$ norm
\begin{align*}
\left(\sum_{k=1}^{N_T} \Delta t \, \|w\|^q_{1,p,\T} \right)^{\frac{1}{q}} \quad \forall w \in \mathcal{H}_\D,
\end{align*}
or with the $L^q(0,T;L^p(\Tor))$ norm
\begin{align*}
\left(\sum_{k=1}^{N_T} \Delta t \, \|w\|^q_{L^p(\Tor)} \right)^{\frac{1}{q}}, \quad \forall w \in \mathcal{H}_\D.
\end{align*}
In particular, in the case $p=q=2$, the $L^2(Q_T)$ norm can also be defined by duality as
\begin{align}\label{3.defdual.norm}
\|w\|^2_{L^2(Q_T)} = \sup \left\lbrace \int_0^T \int_{\Tor} w f \, \mathrm{d}x \mathrm{d}t  \, \, : \, \, f \in \mathcal{H}_{\D}, \, \, \sum_{k=1}^{N_T} \Delta t \sum_{i \in \I} \Delta x |f^k_i|^2 = 1  \right\rbrace, \quad \forall w \in \mathcal{H}_\D.
\end{align}
This dual formulation of $\|\cdot\|_{L^2(Q_T)}$ will be needed later on.

\subsection{Numerical scheme}\label{subsec.scheme}

We discretize the initial conditions \eqref{4.ic} as
\begin{align}\label{5.ic}
u^0_{j,i} = \frac{1}{\Delta x} \int_{K_i} u_j^0(x) \, \mathrm{d}x, \quad \forall i \in \I, \, j=1,2.
\end{align}
Now for given $(u^{k-1}_1,u^{k-1}_2) \in \R^{2N}$, the implicit in time numerical scheme writes as
\begin{align}\label{5.sch}
\frac{u^k_{j,i}-u^{k-1}_{j,i}}{\Delta t} - \left(\Delta_\mathcal{T}  (\mu_j^k u_j^k)\right)_i  = R_j(u^k_{1,i},u^k_{2,j}), \quad \forall i \in \I, \, j=1,2,
\end{align}
where $\Delta_\mathcal{T}$ denotes the discrete Laplacian, namely
\begin{align}\label{1.discreteLapl}
\left(\Delta_\mathcal{T}  (\mu_j^k u_j^k)\right)_i = \frac{\mu^k_{j,i+1} u^k_{j,i+1} - 2 \mu^k_{j,i} u^k_{j,i} + \mu^k_{j,i-1} u^k_{j,i-1}}{\Delta x^2}, \quad \forall i \in \I, \, j=1,2,
\end{align}
and
\begin{align}\label{5.mu1}
\mu_{1,i}^k &= d_1 + d_{11} \sum_{n \in \I} \dx \sigma_{1,i-n} u^k_{1,n}+ d_{12} \sum_{n\in\I} \Delta x \rho_{1,i-n} \, u^k_{2,n}, \quad \forall i \in \I,\\\label{5.mu2}
\mu_{2,i}^k &= d_2 + d_{21} \sum_{n\in\I} \Delta x \rho_{2,i-n} \, u^k_{1,n} + d_{22} \sum_{n \in \I} \dx \sigma_{2,i-n} u^k_{2,n}, \quad \forall i \in \I,
\end{align}
with
\begin{align}\label{5.approx.conv1}
\rho_{j,i-n} = \frac{1}{\dx} \int_{K_{i-n}} \rho_j(y) \, \mathrm{d}y, \quad \sigma_{j,i-n} = \frac{1}{\dx} \int_{K_{i-n}} \sigma_j(y) \, \mathrm{d}y, \quad \forall i, \, n \in \I, \, j=1,2.
\end{align}
Let us notice, by construction, that if we consider $\rho_1=\rho_2=\sigma_1=\sigma_2=\delta_0$ then \eqref{5.ic}--\eqref{5.mu2}  yields a finite volume scheme for the local SKT model \eqref{4.local.SKT1}--\eqref{4.local.SKT2}. We also remark that we could equivalently rewrite \eqref{5.sch} as 
\begin{align}\label{5.sch.div}
\Delta x \frac{u^k_{j,i}-u^{k-1}_{j,i}}{\Delta t} + \mathcal{F}^k_{j,i+\frac{1}{2}} - \mathcal{F}^k_{j,i-\frac{1}{2}} = R_j(u^k_{1,i},u^k_{2,i}), \quad \forall i \in \I, \, j=1,2,
\end{align}
where for all $i \in \I$ the numerical fluxes $\mathcal{F}^k_{j,i+\frac{1}{2}}$ are defined by
\begin{equation}\label{5.sch.flux}
\mathcal{F}^k_{j,i+\frac{1}{2}} = \frac{u^k_{j,i}\mu^k_{j,i}-u^k_{j,i+1}\mu^k_{j,i+1}}{\Delta x} = \mu^k_{j,i+\frac{1}{2}} \frac{u^k_{j,i}-u^k_{j,i+1}}{\Delta x} + u^k_{j,i+\frac{1}{2}} \frac{\mu^k_{j,i}-\mu^k_{j,i+1}}{\Delta x}, \quad j=1,2,
\end{equation}
with the centered approximation at interfaces
\begin{align*}
\mu^k_{j,i+\frac{1}{2}} = \frac{\mu^k_{j,i}+\mu^k_{j,i+1}}{2}, \quad u^k_{j,i+\frac{1}{2}} = \frac{u^k_{j,i}+u^k_{j,i+1}}{2}, \quad j=1,2.
\end{align*}

\subsection{Main results}

Let us collect our assumptions.

\begin{labeling}{(A44)}
\item[(H1)] The domain is taken as $\Omega = \Tor$.
\item[(H2)] The diffusion coefficients $d_1$, $d_2$, $d_{11}$ and $d_{22}$ are non-negative constants and the cross-diffusion coefficients $d_{12}$ and $d_{21}$ are positive constants.
\item[(H3)]  The convolution kernels $\rho_1$, $\rho_2$, $\sigma_1$ and $\sigma_2$ are $L^\infty(\Tor)$ functions that are non-negative and satisfy the symmetry hypotheses \eqref{eq:hypoth}. In particular $\rho_i = \rho_{1,i} =  \rho_{2,-i}$ for all $i\in\I$.
\item[(H4)] The initial data $u^0_1$ and $u^0_2$ are non-negative $L^1(\Tor)$ functions with finite entropy, namely $h_1(u_1^0), h_2(u_2^0)\in L^1(\Tor)$.
\item[(H5)] The reaction terms satisfy $R_1=R_2=0$.
\end{labeling}

As already mentioned, hypothesis (H1) is only made for the convenience of the reader and one can adapt the design of the scheme and the results to a $d$-dimensional periodic domain $\Omega = \prod_{i=1}^d \mathbb{R}/L_i\mathbb{Z}$. 
Observe that by assuming (H2) we require cross-diffusion on both species. While this is crucial in our proofs, the scheme performs well in practice even with $d_{12}=0$  or $d_{21}=0$ (see Section~\ref{sec.num}).  The assumption (H3) on the symmetry of the functions $\rho_1$, $\rho_2$, $\sigma_1$ and $\sigma_2$ are needed, as at the continuous level, to show the discrete entropy inequality satisfied by the solutions of the scheme \eqref{5.ic}--\eqref{5.mu2}. However, in terms of practical use, the scheme performs well even when dropping this hypothesis (see Section~\ref{sec:testcase3}). Following for instance \cite{JZ20}, the assumption (H5) can be relaxed and one can extend the proofs of Theorem \ref{thm.SKTexi} and Theorem \ref{thm.convSKT} in the case of the Lotka-Volterra source terms:
\begin{align*}
    R_j(u_1,u_2) = u_j \, \left(a_{j0} - \sum_{k=1}^2 a_{jk} \, u_k \right), \quad j=1,2,
\end{align*}
with $a_{j0}$ and $a_{jk}$ some nonnegative constants for $j,k=1,2.$

Our first main result deals with the existence of solutions to scheme \eqref{5.ic}--\eqref{5.mu2} at each time step. But first let us recall the definition of the discrete entropy functional
\begin{align*}
H(u^k_1,u^k_2) = \sum_{i \in \I} \Delta x \, h_1(u^k_{1,i}) +\sum_{i \in \I} \Delta x \, h_2(u^k_{2,i}),
\end{align*}
where the functions $h_1$ and $h_2$ are defined by 
\begin{align}\label{6.density.entropy}
h_1(x) =\dfrac{1}{d_{12}} \, \left(x \,(\log(x)-1)+1\right), \quad h_2(x) = \dfrac{1}{d_{21}} \, \left(x \,(\log(x)-1)+1\right), \quad \forall x \in (0,+\infty),
\end{align}
with the obvious continuous extension at $x=0$. The corresponding entropy dissipation functional is defined by 
\begin{multline}\label{6.def.dissip}
D(u^k_1,u^k_2) = 2 \frac{d_{11}}{d_{12}}  \sum_{j \in \I} \sigma_{1,j} \sum_{i\in\I} \left(\sqrt{u^k_{1,i+1} u^k_{1,i+1-j}} - \sqrt{u^k_{1,i} u^k_{1,i-j}} \right)^2 \\
+2 \frac{d_{22}}{d_{21}}  \sum_{j \in \I} \sigma_{2,j} \sum_{i\in\I} \left(\sqrt{u^k_{2,i+1} u^k_{2,i+1-j}} - \sqrt{u^k_{2,i} u^k_{2,i-j}} \right)^2
+ 4 \frac{d_1}{d_{12}}  \Big|\sqrt{u^k_1}\Big|^2_{1,2,\T} \\
+ 4 \frac{d_2}{d_{21}}  \Big|\sqrt{u^k_2}\Big|^2_{1,2,\T}
+4  \sum_{j \in \I} \rho_j \sum_{i \in \I} \left(\sqrt{u^k_{1,i+1} u^k_{2,i+1-j}} - \sqrt{u^k_{1,i} u^k_{2,i-j}} \right)^2.
\end{multline}
\begin{theorem}[Existence of solutions]\label{thm.SKTexi}
Let the assumptions (H1)--(H5) hold. Then, for every $1 \leq k \leq N_T$ there exists (at least) one nonnegative solution $(u^k_1,u^k_2)$ to scheme \eqref{5.ic}--\eqref{5.mu2}. Moreover, this solution satisfies the following properties:
\begin{itemize}
\item[(i)] Mass conservation:
\begin{align}\label{6.conser.mass}
\sum_{i \in \I} \Delta x \, u^k_{j,i} = \int_{\Tor} u^0_j(x) \, \mathrm{d}x, \quad \forall k \geq 0, \, j=1,2.
\end{align}
\item[(ii)] Entropy production estimate: for all $k \geq 1$ it holds
\begin{align}\label{6.entopy.dissip}
H(u^k_1,u^k_2) +\Delta tD(u^k_1,u^k_2)  \leq H(u^{k-1}_1,u^{k-1}_2).
\end{align}
\end{itemize}
Finally  for all $j\in\{1,2\}$, if $d_j$ is positive then $u_j^k$ is positive for all $0<k\leq N_T$.
\end{theorem}

The proof of existence of Theorem \ref{thm.SKTexi} is based on a consequence  (see \cite[Section 9.1]{Evans}) of the Brouwer fixed point Theorem. It can be applied thanks to the \emph{a priori} entropy-dissipation estimate \eqref{6.entopy.dissip} and regularization inspired by \cite{BFPS10, Jue16}. It also follows the line of the existence proof of \cite{chainais2022long}.

The second main result is concerned by some properties satisfied by the solutions of the scheme \eqref{5.ic}--\eqref{5.mu2}. These estimates are discrete counterparts of \cite[Theorem 9]{DM21}.

\begin{theorem}[Qualitative properties of the solutions]\label{thm.prop}
Let the assumptions of Theorem \ref{thm.SKTexi} hold. Moreover, assume that $u^0_1,u^0_2\in L^2(\Tor)$ and
let $\gamma$ and $\Gamma$ be some nonnegative constants such that
\begin{align*}
\gamma \leq u^0_1(x), \, u^0_2(x) \leq \Gamma \quad \mbox{a.e. }x \in\Tor.
\end{align*}
Finally let us introduce $m^0_j = \|u^0_j\|_{L^1(\Tor)}$ for $j=1,2$. Then the following properties hold. 
\begin{itemize}
\item[(i)] Maximum principle: If $\rho$, $\sigma_1$ and $\sigma_2$ are twice continuously differentiable functions and that the time step satisfies the condition
\begin{align*}
\Delta t < 1/\left(\min\lbrace d_{11} m^0_1 \|\Delta  \sigma_1\|_{L^\infty(\Tor)},d_{22} m^0_2 \|\Delta  \sigma_2\|_{L^\infty(\Tor)} \rbrace + \min\lbrace d_{12} m^0_2, d_{21} m^0_1\rbrace \|\Delta  \rho\|_{L^\infty(\Tor)} \right),
\end{align*} 
then for all $i \in \I$, $k \geq 1$ and $j=1,2$ we have 
\begin{align*}
e_k \leq u^k_{j,i} \leq E_k,
\end{align*}
where
\begin{align*}
e_k &= \gamma \left(1+ \Delta t\left( \min\lbrace d_{11} m^0_1\|\Delta  \sigma_1\|_{L^\infty(\Tor)},d_{22}m^0_2 \|\Delta  \sigma_2\|_{L^\infty(\Tor)} \rbrace + \min\lbrace d_{12} m^0_2, d_{21} m^0_1\rbrace \|\Delta  \rho\|_{L^\infty(\Tor)} \right)\right)^{-k},\\ 
E_k &= \Gamma \left(1-\Delta t \left(\min\lbrace d_{11} m^0_1\|\Delta \sigma_1\|_{L^\infty(\Tor)},d_{22} m^0_2 \|\Delta \sigma_2\|_{L^\infty(\Tor)} \rbrace + \min\lbrace d_{12}m^0_2, d_{21} m^0_1\rbrace \|\Delta  \rho\|_{L^\infty(\Tor)} \right)\right)^{-k}.
\end{align*}
\item[(ii)] Duality estimate: If $d_1$ and $d_2$ are positives, then there exists a constant $C>0$ which is independent of the mesh size such that
\begin{align*}
\sum_{k=1}^{N_T} \Delta t \sum_{i \in \I} \Delta x \left(\mu^k_{1,i} u^k_{1,i} + \mu^k_{2,i} u^k_{2,i} \right) \left(u^k_{1,i} + u^k_{2,i} \right) \leq  C (1+TA ) \left(\|u^0_{1}\|_{L^2(\Tor)}^2 + \|u^0_{2}\|_{L^2(\Tor)}^2 \right),
\end{align*}
where
\begin{align*}
A=d_1+d_2+d_{11} m^0_1 \, \|\sigma_1\|_{L^1(\Tor)}+d_{22} m^0_2 \, \|\sigma_2\|_{L^1(\Tor)}+\|\rho\|_{L^1(\Tor)}\left(d_{12} m^0_2 + d_{21}  m^0_1\right).
\end{align*}
\end{itemize}
\end{theorem}

The proof of Theorem \ref{thm.prop} relies on a discrete duality method. In Section \ref{sec.Kolmo}, we define and study the properties satisfied by the finite volume solutions to the Kolmogorov equation. Then, in Section \ref{sec.proof.thm.prop},  we apply these results on the solutions to the scheme \eqref{5.ic}--\eqref{5.mu2} in order to establish the Theorem. Let us emphasize that the duality estimate holds without any assumptions on the time step or the regularity of the convolution kernels. This implies in particular that this discrete estimate also holds for the solutions to the local SKT system.

Finally, we show the convergence of the solutions to the scheme \eqref{5.ic}--\eqref{5.mu2} towards a distributional solution to \eqref{4.SKT1}--\eqref{4.ic} in the sense of Definition \ref{def.distrib.solution}. However, in order to state precisely our convergence result, we need some notations.

We introduce a family $(\mathcal{D}_m)_{m\in\N}$ of space-time discretizations of $Q_T$ indexed by the size 
$\eta_m=\max\{\Delta x_m,\Delta t_m\}$ of the mesh, satisfying $\eta_m\to 0$
as $m\to\infty$. We denote by $\T_m$ the corresponding mesh of $\Tor$ and
by $\Delta t_m$ the corresponding time step. Finally, for every $m \in \N$ we set $(u_{1,m},u_{2,m}) \in \mathcal{H}_{\mathcal D_m}$ the picewise constant in space and time reconstruction of the solutions to the scheme \eqref{5.ic}--\eqref{5.mu2} corresponding to the mesh $\D_m$. 

\begin{theorem}[Convergence of the scheme]\label{thm.convSKT}
Let the assumptions of Theorem \ref{thm.SKTexi} hold, assume that the coefficients $d_1$ and $d_2$ are positives and let $(\D_m)_{m \in \N}$ be a family of space-time discretizations of $Q_T$ with $\eta_m \to 0$ as $m \to \infty$. Then, if we denote by $(u_{1,m},u_{2,m})$ a family of finite volume solutions to \eqref{5.ic}--\eqref{5.mu2} obtained in Theorem \ref{thm.SKTexi}, there exists $(u_1,u_2) \in \left(L^p(Q_T) \right)^2$ for $p \in [1,3)$ a distributional solutions to \eqref{4.SKT1}--\eqref{4.ic} in the sense of Definition \ref{def.distrib.solution} such that, up to a subsequence, for $j=1,2$ it holds
\begin{align*}
u_{j,m} \to u_j \quad\mbox{strongly in }L^p(Q_T) \mbox{ for }1 \leq p < 3  \quad \mbox{as }m\to \infty.
\end{align*}
\end{theorem}

The proof of Theorem \ref{thm.convSKT} is based on uniform estimates w.r.t. $\dx$ and $\Delta t$, established in Section \ref{sec.unif}. These estimates allow us to apply in Section \ref{subsec.compactness} a compactness result obtained in \cite{ABR11} which yields, up to a subsequence, the strong convergence in $L^p(Q_T)$ of the sequence $(u_{1,m},u_{2,m})$ towards the functions $u_1$ and $u_2$ stated in Theorem \ref{thm.convSKT}. Then, we identify in Section \ref{subsec.proof.conv} the functions $u_1$ and $u_2$ as distributional solutions in the sense of Definition \ref{def.distrib.solution} of the nonlocal cross-diffusion system \eqref{4.SKT1}--\eqref{4.ic}.

\begin{remark}
Let us make few remarks concerning Theorem \ref{thm.convSKT}.
\begin{itemize}
\item If $d > 1$, the convergence of the scheme can also be established. However in this case we obtain, up to a subsequence, for $j=1,2$,
\begin{align*}
u_{j,m} \rightarrow u_j \quad \mbox{strongly in }L^p(Q_T) \mbox{ for }1 \leq p < {\frac{d}{d-1}}, \quad \mbox{as }m \to \infty,
\end{align*}
see Remark \ref{rem.dimsupconv} for more details. 
\item If  $d_1=d_2 = 0$, then it is still possible to conclude if the convolution kernels are smooth enough ($C^2$ for instance). Indeed in this case from weak compactness on $(u_{j,m})_m$, strong compactness can be obtained on its convolution with the smooth kernel. Moreover, if the convolution kernels are $C^2$, then one can prove a stability estimate (in $L^2$-norm on the difference between two solutions) for the solutions to \eqref{4.SKT1}--\eqref{4.ic} which provides uniqueness and continuous dependence on the initial data at the continuous level. As a by-product we deduce that in this case the whole sequence $(u_{1,m}, u_{2,m})$ converges as $m \to \infty$ instead of only a subsequence. In the discrete setting, a  counterpart of the stability estimate in $L^2$-norm can also be established uniformly in $\Delta x$ at the price of additional $H^1$ regularity on the initial data. This difference with the continuous setting comes from the fact that a Gr\"onwall argument with implicit schemes requires a condition on the time step (see  Proposition~\ref{prop.estimate_kolmo} for an illustration of this fact). Under these assumptions, one also gets uniqueness of solutions to the scheme. Uniqueness may also be obtained without additional regularity assumptions provided that a CFL condition holds (see Remark~\ref{rem:CFL} for details). 
\item Finally, with enough regularity on the data, one could derive quantitative error estimates between approximate and continuous solutions.
\end{itemize}
\end{remark}

\section{Existence of solution and entropy dissipation estimate}\label{subsec.proofexiSKT}

The problem of existence of solution reduces to the resolution of a nonlinear system of equations. The natural unknowns for which a fixed point theorem will be easily applied are linked to the entropy. In our case, given $(u_1,u_2)\in((0,+\infty)^N)^2$ we define the new unknown $X = \Phi(u_1,u_2)$ where $\Phi:((0,+\infty)^N)^2\to\R^{2N}$ is the smooth diffeomorphism defined by
\[
\Phi(u_1,u_2) = (d_{12}^{-1}\log(u_{1,1}),\dots, d_{12}^{-1}\log(u_{1,N}), d_{21}^{-1}\log(u_{2,1}),\dots, d_{21}^{-1}\log(u_{2,N}))^\top\in\R^{2N}.
\]
From there finding a positive solution to the scheme \eqref{5.ic}--\eqref{5.mu2} amounts to finding a zero $X^k = \Phi(u_1^k,u_2^k)$ of the continuous map $P^k:\R^{2N}\to\R^{2N}$ defined for any $(u_1,u_2)\in((0,+\infty)^N)^2$ by its components
\[
P_{i + N(j-1)}^k(\Phi(u_1,u_2)) = \Delta x(u_{j,i}-u^{k-1}_{j,i}) - \Delta t\Delta x\left(\Delta_\mathcal{T}  (\mu_j u_j)\right)_i,\quad \forall i\in\{1,\dots,N\}, j\in\{1,2\},
\]
where $(u_1^{k-1},u_2^{k-1})$ are given and $\mu_1,\mu_2$ are related to $u_1,u_2$ through the relation \eqref{5.mu1} and \eqref{5.mu2} dropping the index $k$. 
\subsection{Entropy dissipation and mass conservation}
In the following $\lla,\rra$ denotes the Euclidean scalar product and $|\cdot|$ the Euclidean norm.
\begin{proposition}\label{prop:Pproperties}
Let $(u_1^{k-1},u_2^{k-1})$ be componentwise non-negative. Then for any $X\in\R^{2N}$,
\begin{equation}\label{eq:Pconservmass}
\lla P^k(X),\mathds{1}_j\rra = \sum_{i\in\I}(u_{j,i}-u^{k-1}_{j,i})\Delta x \,,\quad \forall j \in\{1,2\},
\end{equation}
and
\begin{equation}\label{eq:Pentropy}
\lla P^k(X),X\rra\geq H(u_1,u_2) - H(u_1^{k-1},u_2^{k-1}) + \Delta tD(u_1,u_2),
\end{equation}
where $\mathds{1}_1 = (1,\dots,1,0,\dots,0)$, $\mathds{1}_2 = (0,\dots,0,1,\dots,1)$, $(u_1,u_2) = \Phi^{-1}(X)$ and $D(u_1,u_2)$ denotes the entropy dissipation functional given by \eqref{6.def.dissip}. 
\end{proposition}
\begin{proof}
In order to prove \eqref{eq:Pconservmass}, it suffices to sum the components of $P^k(X)$ and observe that
\[
\sum_{i\in\I}\left(\Delta_\mathcal{T}(\mu_j u_j)\right)_i  = 0,
\]
since it is a telescopic sum. Concerning the inequality, first observe that $H(u_1^{k-1},u_2^{k-1})$ is well-defined since $(u_1^{k-1},u_2^{k-1})$ is non-negative. Then, using the definition of $\Phi(u_1,u_2)$ and $P^k$ one obtains
\[
\lla P^k(X),X\rra = \Delta t(I_1 + I_2 + J_1 + J_2)
\]
with 
\begin{align*}
 I_1 &= \frac{1}{d_{12}}\sum_{i \in \I} \dx (u_{1,i}-u^{k-1}_{1,i}) \, \log(u_{1,i}),\\
 I_2 &= \frac{1}{d_{21}}\sum_{i \in \I} \dx (u_{2,i}-u^{k-1}_{2,i}) \, \log(u_{2,i}),\\
 J_1 &= \frac{1}{d_{12}\dx} \sum_{i \in \I}(-\mu_{1,i+1} u_{1,i+1} + 2 \mu_{1,i} u_{1,i} - \mu_{1,i-1} u_{1,i-1}) \, \log(u_{1,i})\\
 J_2 &= \frac{1}{d_{21}\dx} \sum_{i \in \I} (-\mu_{2,i+1} u_{2,i+1} + 2 \mu_{2,i} u_{2,i} - \mu_{2,i-1} u_{2,i-1}) \, \log(u_{2,i})
\end{align*}
Using the convexity of $x \mapsto (x\log(x)-x+1)$ to bound both $I_1$ and $I_2$ from below, one obtains
\[
I_1 + I_2\geq H(u_1,u_2) - H(u_1^{k-1},u_2^{k-1}).
\]
Then for $J_1$, a discrete integration by parts (or summation by parts) yields
\[
J_1 = \frac{1}{d_{12} \,\Delta x} \sum_{i \in \I} \left(u_{1,i+1} \mu_{1,i+1} - u_{1,i} \mu_{1,i} \right) \left(\log(u_{1,i+1}) - \log(u_{1,i}) \right),
\]
and a similar formula holds for $J_2$. Using the definitions of $\mu_{1}$ and $\mu_{2}$ (see \eqref{5.mu1} and \eqref{5.mu2} without the exponents), one has $J_1 = J_{1}^\text{diff}+J_{1}^{\sigma_1}+J_{1}^{\rho_1}$ and $J_2 = J_{2}^\text{diff}+J_{2}^{\sigma_2}+J_{2}^{\rho_2}$ with
\begin{align*}
J_{1}^\text{diff} &= \frac{d_1}{d_{12}\dx} \sum_{i \in \I} \left(u_{1,i+1}-u_{1,i} \right) \left(\log(u_{1,i+1})-\log(u_{1,i}) \right)\geq\frac{4d_1}{d_{12}\dx} \sum_{i \in \I} \left(\sqrt{u_{1,i+1}}-\sqrt{u_{1,i}} \right)^2,
\end{align*}
and a similar estimate for $J_2^\text{diff}$. For the second term one has
\begin{align*}
J_{1}^{\sigma_1} &= \frac{d_{11}}{d_{12}} \sum_{j\in\I}\sigma_{1,j}\sum_{i \in \I} \left( u_{1,i+1}   u_{1,i+1-j} - u_{1,i}  u_{1,i-j} \right) \left(\log(u_{1,i+1})-\log(u_{1,i}) \right)\\
&= \frac{1}{2}J_{1}^{\sigma_1} + \frac{d_{11}}{2d_{12}} \sum_{j\in\I} \sigma_{1,j} \sum_{i \in \I} \left(u_{1,i+1} u_{1,i+1+j} - u_{1,i} u_{1,i+j} \right) \left(\log(u_{1,i+1}) - \log(u_{1,i}) \right)\\
&= \frac{1}{2}J_{1}^{\sigma_1} + \frac{d_{11}}{2d_{12}} \sum_{j\in\I} \sigma_{1,j} \sum_{i \in \I} \left(u_{1,i+1-j} u_{1,i+1} - u_{1,i-j} u_{1,i} \right) \left(\log(u_{1,i+1-j}) - \log(u_{1,i-j}) \right)\\
&= \frac{d_{11}}{2d_{12}} \sum_{j\in\I} \sigma_{1,j} \sum_{i \in \I} \left( u_{1,i+1}   u_{1,i+1-j} - u_{1,i}  u_{1,i-j} \right) \left(\log(u_{1,i+1}u_{1,i+1-j}) - \log(u_{1,i}u_{1,i-j}) \right)\\
&\geq \frac{2d_{11}}{d_{12}} \sum_{j\in\I} \sigma_{1,j} \sum_{i \in \I} \left( \sqrt{u_{1,i+1}   u_{1,i+1-j}} - \sqrt{u_{1,i}  u_{1,i-j}} \right)^2.
\end{align*}
 In the previous estimate, the second inequality is obtained by changing $j$ into $-j$ and using the symmetry of $\sigma_1$. For the third equality, one changes $i$ into $i-j$. The fourth one is the combination of the first and third equalities. Once again a similar estimate holds for $J_{2}^{\sigma_2}$. Finally with the same changes of indices one can estimate the sum
\begin{align*}
J_{1}^{\rho_1} + J_{2}^{\rho_2} =&  \sum_{j \in \I} \rho_j\sum_{i \in \I} \left(u_{1,i+1}  u_{2,i+1-j} - u_{1,i} u_{2,i-j} \right) \left(\log(u_{1,i+1})-\log(u_{1,i}) \right)\\
&+ \sum_{j \in \I} \rho_{-j}\sum_{i \in \I} \left(u_{2,i+1}  u_{1,i+1-j} - u_{2,i} u_{1,i-j} \right) \left(\log(u_{2,i+1})-\log(u_{2,i}) \right)\\
=&  \sum_{j \in \I} \rho_j\sum_{i \in \I} \left(u_{1,i+1}  u_{2,i+1-j} - u_{1,i} u_{2,i-j} \right) \left(\log(u_{1,i+1})-\log(u_{1,i}) \right)\\
&+ \sum_{j \in \I} \rho_{j}\sum_{i \in \I} \left(u_{2,i-j+1}  u_{1,i+1} - u_{2,i-j} u_{1,i} \right) \left(\log(u_{2,i-j+1})-\log(u_{2,i-j}) \right)\\
=&  \sum_{j \in \I} \rho_j\sum_{i \in \I} \left(u_{1,i+1}  u_{2,i+1-j} - u_{1,i} u_{2,i-j} \right) \left(\log(u_{1,i+1}u_{2,i-j+1})-\log(u_{1,i}u_{2,i-j}) \right)\\
\geq&  4\sum_{j \in \I} \rho_j\sum_{i \in \I} \left(\sqrt{u_{1,i+1}  u_{2,i+1-j}} - \sqrt{u_{1,i} u_{2,i-j}} \right)^2.
\end{align*}
By summing all the estimates one obtains \eqref{eq:Pentropy}. The last point of the proposition is obtained by induction.
\end{proof}

\subsection{Proof of Theorem \ref{thm.SKTexi}}
Let us show that $P^k\circ\Phi$ has at least one zero. We use an approximation argument by introducing 
\[
P_\varepsilon^k(X) = P^k(X)+\varepsilon X,\quad \forall X\in\R^{2N}\,.
\]
Using \eqref{eq:Pentropy} and the non-negativity of the entropy and the entropy dissipation one has
\[
\lla P^k_\varepsilon(X),X\rra\geq \varepsilon|X|^2 - H(u_1^{k-1},u_2^{k-1})\,.
\]
Therefore, as a consequence of Brouwer fixed point theorem (see \cite[Section 9.1]{Evans} for details), there is $X_\varepsilon$ such that 
\[
 P^k_\varepsilon(X_\varepsilon) = 0\quad \text{and} \quad |X_\varepsilon|^2\leq H(u_1^{k-1},u_2^{k-1})\varepsilon^{-1}\,.
\]
Let us define the associated $(u_{1}^\varepsilon,u_{2}^\varepsilon) = \Phi^{-1}(X_\varepsilon)$, which is componentwise positive by definition. Observe that 
 \begin{align*}
   \dx \, h_j(u_{j,i}^{\eps})
 	\le H(u_1^{\eps},u^\eps_2) \le  H(u^{k-1}_1,u^{k-1}_2), \quad \forall i \in \I, \, j=1,2.
 \end{align*}
where the last inequality is again a consequence of \eqref{eq:Pentropy} for $X=X_\varepsilon$. This shows that for any $(u_{1,i}^{\eps},u_{2,i}^\eps)$ is uniformly bounded in $\varepsilon$. Therefore, there exists a subsequence (not relabeled) such that $u_{j,i}^{\eps}\to u_{j,i}^k \geq 0$ as $\eps\to 0$, for every $i \in \I$ and $j=1,2$. Since $|X_\varepsilon| = O(\varepsilon^{-1/2})$ one has 
\[
0 = \lim_{\varepsilon\to0}P^k_\varepsilon(X_\varepsilon) = P^k(\Phi(u_1^k,u_2^k))
\]
Therefore $(u_1^k,u_2^k)$ solves the scheme \eqref{5.ic}-\eqref{5.mu2}. By taking limits in \eqref{eq:Pconservmass} and \eqref{eq:Pentropy} evaluated at $X_\varepsilon$ as $\varepsilon\to0$ one recovers \eqref{6.conser.mass} and \eqref{6.entopy.dissip} respectively. 

Let us finally prove that if $d_j>0$, $u^k_{j,i}> 0$ for all $i \in \I$  and $0<k\leq N_T$. This is a consequence of the entropy estimate. For a given $0<k\leq N_T$ we notice that (thanks to the term $J^{\mathrm{diff}}_j$) the positive solution $(u^\varepsilon_1,u^\varepsilon_2)$ satisfies the following estimate
\begin{align*}
    \frac{d_j \, \Delta t}{\dx}\sum_{i \in \I} \left(u^\varepsilon_{j,i+1}-u^\varepsilon_{j,i}\right)\, \left(\log\left(u^\varepsilon_{j,i+1}\right)-\log\left(u^\varepsilon_{j,i}\right)\right) \leq \max(d_{12},d_{21}) \, H\left(u^{k-1}_1,u^{k-1}_2\right) \quad j=1,2.
\end{align*}
At the limit $\varepsilon\to0$, let us assume by contradiction that there exists $i \in \I$ such that $u^k_{j,i}=0$. Then as the r.h.s. of the previous inequality is finite this implies that $u^k_{j,i+1}=0$. Thus repeating this argument we deduce that $u^k_{j,i}=0$ for all $i \in \I$. Consequently we have $\|u^k_j\|_{L^1(\Tor)}=0$ which contradicts the mass conservation property \eqref{6.conser.mass}. This completes the proof of Theorem~\ref{thm.SKTexi}.

\begin{remark}[Uniqueness under CFL]\label{rem:CFL}
Under a parabolic CFL condition, uniqueness of a solution to the scheme can also be proven. Indeed, if one denotes by $U^k$ a vector of solution given by Theorem~\ref{thm.SKTexi}, then the scheme may be rewritten as $\tilde{\mathbb{M}}^k(U^k) U^k = U^{k-1}$. The matrix $\tilde{\mathbb{M}}^k$ (which has a similar definition as \eqref{2.defMk} hereafter) depends on $U^k$ through $\mu_1^k$ and $\mu_2^k$ and is a perturbation of the identity matrix of size $\frac{ \Delta t}{\Delta x^2}(\|\mu_1^k\|_{L^\infty} + \|\mu_2^k\|_{L^\infty})$. Since $\|\mu_i^k\|_{L^\infty}$ is bounded uniformly in terms of the initial masses, $L^\infty$ norms of the convolution kernels and diffusion coefficients, $\frac{ \Delta t}{\Delta x^2}$ can be made small enough with respect to these quantities only so that a contraction argument yields the uniqueness of solutions. 
\end{remark}

\section{Estimates on the discrete Kolmogorov equation}\label{sec.Kolmo}

In this section, we focus on estimates concerning the finite volume discretization of the Kolmogorov equation $\partial_t z = \Delta(\mu z)$. In particular we adapt at the discrete level some properties established in \cite{AM20,DM21}. 

In the rest of this section, we assume that $(\mu^k_i)_{i\in\I}$, $k=1,\dots,N_T$ is given and componentwise non-negative. From there, the scheme is given for all $k\geq1$ by
\begin{align}\label{1.sch}
 \frac{z^k_i-z^{k-1}_i}{\Delta t} - \left(\Delta_\mathcal{T}  (\mu^k z^k)\right)_i  = 0, \quad \forall i \in \I,
\end{align}
where $\Delta_\mathcal{T} $ denotes the discrete Laplacian operator defined by \eqref{1.discreteLapl}.

\subsection{Well-posedness of the scheme and $L^\infty$ estimates}
Let us first prove that the scheme \eqref{1.sch} admits a unique solution at each time step.
\begin{lemma}\label{lem.exi}
For any $(z^{k-1}_i)_{i\in\I}$ there is a unique $(z^{k}_i)_{i\in\I}$ satisfying \eqref{1.sch}. Moreover, if $z^{k-1}$ is componentwise nonnegative then so is $z^{k}$.
\end{lemma}
\begin{proof}
Let us write $Z^{k-1}=(z^{k-1}_0,\ldots,z^{k-1}_{N-1})^\top$ for all $k\geq1$. Observe that the scheme writes
$\mathbb{M}^k Z^k = Z^{k-1}$ where $\mathbb{M}^k$ is a $N \times N$ tridiagonal matrix defined by
\begin{align}\label{2.defMk}
\M^k_{i,i-1} = - \frac{\Delta t}{\Delta x^2}\mu^k_{i-1}, \quad \M^k_{i,i} = 1+2  \frac{\Delta t}{\Delta x^2}\mu^k_i, \quad \M^k_{i,i+1} = -  \frac{\Delta t}{\Delta x^2}\mu^k_{i+1}, \quad \forall i\in \I.
\end{align}
We notice that $\M^k$ has positive diagonal terms and non-positive off-diagonal terms. Furthermore the matrix $\M^k$ is strictly diagonally dominant with respect to its columns. Therefore $\M^k$ is a non-singular M-matrix and is thus monotone and invertible. This finishes the proof of Lemma~\ref{lem.exi}.
\end{proof}

We prove in the following result some $L^\infty$ estimates for the solution to scheme \eqref{1.sch}.

\begin{lemma}\label{lem.Linftybounds}
 Let us assume that there exists $\tilde{\gamma}, \tilde{\Gamma} \geq 0$ such that
 \begin{align*}
 \tilde{\gamma} \leq z^0_i \leq \tilde{\Gamma}, \quad \forall i\in\I.
 \end{align*}
 Then for every $k \geq 1$ and every $\Delta t > 0$ such that
\begin{align*}
 \Delta t < 1/ \max_{1 \leq k \leq N_T} \|[\Delta_\mathcal{T}  \mu^k]_+\|_{L^\infty(\Tor)},
 \end{align*}
 the solution $Z^k$ to \eqref{1.sch} satisfies
 \begin{align}\label{2.Linftybounds}
 \tilde{\gamma} \, \Pi_{n=1}^k \left(1+ \Delta t \|[\Delta_\mathcal{T}  \mu^n]_-\|_{L^\infty(\Tor)} \right)^{-1} \leq z^k_i \leq \tilde{\Gamma} \, \Pi_{n=1}^k \left(1-\Delta t \|[\Delta_\mathcal{T}  \mu^n]_+\|_{L^\infty(\Tor)} \right)^{-1}, \,\,\, \forall i \in \I,
 \end{align}
 where $[x]_+=\max(x,0)$ and $[x]_-=\min(x,0)$.
\end{lemma}
\begin{proof}
We will only deal with the upper bound in \eqref{2.Linftybounds} and the lower bound is obtained in the same way. Let $\M^k$ denote the tridiagonal matrix defined by \eqref{2.defMk} and define 
\[
\tilde{\Gamma}^k = \tilde{\Gamma} \, \Pi_{n=1}^k \left(1-\Delta t \|[\Delta_\mathcal{T}  \mu^n]_+\|_{L^\infty(\Tor)} \right)^{-1}.
\]
We proceed by induction. Since $\tilde{\Gamma}^0= \tilde{\Gamma}$ the bound holds by hypothesis at $k=0$. Then observe that for every $i \in \I$
 \begin{align*}
 (\M^k (Z^k-\tilde{\Gamma}^k))_i = Z^{k-1}-\tilde{\Gamma}^k + \Delta t \tilde{\Gamma}^k \, (\Delta_\mathcal{T}  \mu^k)_i\leq \tilde{\Gamma}^{k-1} - \tilde{\Gamma}^k + \Delta t \tilde{\Gamma}^k \, (\Delta_\mathcal{T}  \mu^k)_i.
 \end{align*}
 Now we notice that by construction
 \begin{align*}
 \tilde{\Gamma}^{k-1}-\tilde{\Gamma}^k = - \Delta t \tilde{\Gamma}^k \, \|[\Delta_\mathcal{T}  \mu^k]_+\|_{L^{\infty}(\Tor)}.
\end{align*}
 Then we easily deduce that for every $i \in \I$ it holds
 \begin{align*}
 (\M^k (Z^k-\tilde{\Gamma}^k))_{i} = - \Delta t \tilde{\Gamma}^k \left( \|[\Delta_\mathcal{T}  \mu^k]_+\|_{L^{\infty}(\Tor)} - (\Delta_\mathcal{T}  \mu^k)_i \right) \leq 0.
 \end{align*}
 Therefore, since $\M^k$ is a M-matrix we conclude that $z^k_i \leq \tilde{\Gamma}^k$ for all $i \in \I$ which concludes the proof of Lemma \ref{lem.Linftybounds}.
\end{proof}

The bounds of Lemma \ref{lem.Linftybounds} are exactly the discrete equivalent of the $L^\infty$ estimates established at the continuous level in \cite[Corollary 18]{DM21}.

\begin{proposition}\label{prop.estimate_kolmo}
Let us assume that it holds
\begin{align*}
\Delta t < 1/\max_{1 \leq k \leq N_T} \|[\Delta_\mathcal{T}  \mu^k]_+\|_{L^{\infty}(\Tor)}.
\end{align*}
Then the solution to \eqref{1.sch} satisfies the following estimate
\begin{multline*}
\|z^k\|_{L^2(\Tor)}^2 + \sum_{n=1}^{k} \Delta t \sum_{i \in \T} (\mu^n_i+\mu^n_{i+1}) \, \frac{(z^n_{i+1}-z^n_i)^2}{\Delta x} \\ 
\leq \Pi_{n=1}^k \left(1- \Delta t \|[\Delta_\mathcal{T}  \mu^n]_+\|_{L^{\infty}(\Tor)} \right)^{-1} \, \|z^0_i\|_{L^2(\Tor)}^2 , \quad \forall 1 \leq k \leq N_T.
\end{multline*}
\end{proposition}

\begin{proof}
Let $k \geq 1$ be fixed and let us first notice that we can rewrite for every $i \in \I$ equation \eqref{1.sch} as
\begin{align*}
\Delta x \frac{z^k_i-z^{k-1}_i}{\Delta t} + \mu^k_{i+\frac{1}{2}} \frac{(z^k_i-z^k_{i+1})}{\Delta x} - \mu^k_{i-\frac{1}{2}} \frac{(z^k_{i-1}-z^k_i)}{\Delta x} + z^k_{i+\frac{1}{2}} \frac{(\mu^k_i-\mu^k_{i+1})}{\Delta x} - z^k_{i-\frac{1}{2}} \frac{(\mu^k_{i-1}-\mu^k_i)}{\Delta x}=0,
\end{align*}
 where
\begin{align*}
\mu^k_{i+\frac{1}{2}} = \frac{\mu^k_i+\mu^k_{i+1}}{2}, \quad z^k_{i+\frac{1}{2}} = \frac{z^k_i+z^k_{i+1}}{2}, \quad \forall i \in \I.
\end{align*}
Now we multiply the above equation by $\Delta t z^k_i$ and we sum over $i \in \I$, we obtain
\begin{align*}
I_3 + I_4 + I_5=0,
\end{align*}
where
\begin{align*}
I_3 &= \sum_{i \in \I} \Delta x (z^k_i-z^{k-1}_i) z^k_i,\\
I_4 &= \Delta t \sum_{i \in \I} \left( \mu^k_{i+\frac{1}{2}} \frac{(z^k_i-z^k_{i+1})}{\Delta x} - \mu^k_{i-\frac{1}{2}} \frac{(z^k_{i-1}-z^k_i)}{\Delta x}\right) z^k_i,\\
I_5 &= \Delta t \sum_{i \in \I} \left( z^k_{i+\frac{1}{2}} \frac{(\mu^k_i-\mu^k_{i+1})}{\Delta x} - z^k_{i-\frac{1}{2}} \frac{(\mu^k_{i-1}-\mu^k_i)}{\Delta x} \right) z^k_i.
\end{align*}
For $I_3$ using the inequality $(a-b)a \geq (a^2-b^2)/2$ we obtain
\begin{align}\label{2.I1}
I_3 \geq \frac{1}{2} \sum_{i \in \I} \Delta x \left(|z^k_i|^2 - |z^{k-1}_i|^2\right).
\end{align}
For $I_4$ applying a discrete integration by parts yields
\begin{align}\label{2.I2}
I_4 = \Delta t \sum_{i \in \I} \mu^k_{i+\frac{1}{2}} \, \frac{(z^k_{i+1}-z^k_i)^2}{\Delta x}.
\end{align}
Now we rewrite $I_5$ as
\begin{align*}
I_5 = -\frac{\Delta t}{2} \sum_{i \in \I} \Delta x |z^k_i|^2 (\Delta_\mathcal{T}  \mu^k)_i +\frac{\Delta t}{2\Delta x} \sum_{i \in \I} \left( z^k_{i+1} z^k_i (\mu^k_i - \mu^k_{i+1}) - z^k_{i-1} z^k_i ( \mu^k_{i-1}-\mu^k_i) \right),
\end{align*}
and reordering the terms in the r.h.s. the second sum vanishes and we have
\begin{align}\label{2.I3}
I_5 = -\frac{\Delta t}{2} \sum_{i \in \I} \Delta x |z^k_i|^2 (\Delta_\mathcal{T}  \mu^k)_i.
\end{align}
Gathering \eqref{2.I1}--\eqref{2.I3} we end up with
\begin{align*}
\frac{1}{2} \sum_{i \in \I} \Delta x |z^k_i|^2 + \Delta t \sum_{i \in \I} \mu^k_{i+\frac{1}{2}} \frac{(z^k_{i+1}-z^k_i)^2}{\Delta x} \leq \frac{1}{2} \sum_{i \in \I} \Delta x |z^{k-1}_i|^2  + \frac{\Delta t}{2} \|[\Delta_\mathcal{T}  \mu^k]_+\|_{L^{\infty}(\Tor)}  \sum_{i \in \I} \Delta x |z^k_i|^2.
\end{align*}
We deduce that
\begin{multline*}
\frac{1}{2} \sum_{i \in \I} \Delta x |z^k_i|^2 + \sum_{n=1}^k \Delta t \sum_{i \in \I} \mu^n_{i+\frac{1}{2}} \frac{(z^n_{i+1}-z^n_i)^2}{\Delta x} \\
\leq \frac{1}{2} \sum_{i \in \I} \Delta x |z^0_i|^2  + \sum_{n=1}^k \frac{\Delta t}{2} \|[\Delta_\mathcal{T}  \mu^n]_+\|_{L^{\infty}(\Tor)}  \sum_{i \in \I} \Delta x |z^n_i|^2.
\end{multline*}
One ends the proof of Proposition \ref{prop.estimate_kolmo} thanks to a discrete Gr\"onwall inequality.
\end{proof}

\subsection{Study of the dual problem}

The main objective of this section is to establish a discrete counterpart of the so-called duality inequality for the solution to \eqref{1.sch}, see for instance \cite[Theorem 3]{AM20}. In this aim, following \cite{AM20}, we introduce a ``dual'' scheme associated to \eqref{1.sch}. Let $v^{N_T+1}_i$ be given for every $i \in \I$, then for $1\leq k \leq N_T$ we want to determine the solution to the following implicit backward in time scheme
\begin{align}\label{3.sch.dual}
 \frac{v^k_i-v^{k+1}_i}{\Delta t} - \mu^k_i \left(\Delta_\mathcal{T}  v^k \right)_i 	=  S^k_i \quad \forall i \in \I,
\end{align}
 where $\mu^k_i$ is given and non-negative and $S^k=(S^k_0,\ldots,S^k_{N-1})$ is some given vector in $\R^N$ for all $1 \leq k \leq N_T$. Let us notice that \eqref{3.sch.dual} define a set of linear equation which can be rewritten as
\begin{align}\label{3.pb.dual}
(\mathbb{M}^k)^\top V^k = V^{k+1}+\Delta t S^k, \quad \forall 1 \leq k \leq N_T,
\end{align}
where $\mathbb{M}^k$ is the tridiagonal matrix given by \eqref{2.defMk}. Therefore, it follows directly from the proof of Lemma \ref{lem.exi} that the problem \eqref{3.pb.dual} admits a unique solution for every $1 \leq k \leq N_T$.

Prior to the proof of the discrete duality estimate, see Theorem \ref{thm.duality} below, we establish some uniform estimates satisfied by the solution of \eqref{3.pb.dual}.

\begin{proposition}\label{prop.estimates_pbdual}
Assume that $\min_{i \in \I} \mu^k_i > 0$ for every $0 \leq k \leq N_T$ and that $v^{N_T+1}_i=0$ for every $i \in \I$. Then the solution to \eqref{3.pb.dual} satisfies for every $1 \leq k \leq N_T$ the following estimate
\begin{align}\label{3.estim_pbdual1}
|v^k|^2_{1,2,\T} + \sum_{n=k}^{N_T} \Delta t \sum_{i \in \I} \mu^n_i (\Delta_\mathcal{T}  v^n)^2_i\Delta x
\leq \|\mu^{-1/2} S\|^2_{L^2(Q_T)},
 \end{align}
 and there exists a constant $C>0$ independent of $\Delta x$ such that
 \begin{align}\label{3.estim_pbdual2}
 \|v^k\|^2_{L^2(\Tor)} \leq C (1 + \|\mu\|_{L^1(Q_T)}) \, \|\mu^{-1/2} S\|^2_{L^2(Q_T)}, \quad \forall 1 \leq k \leq N_T,
 \end{align}
 where $\mu$ and $S$ denote the piecewise reconstruction functions in $\mathcal{H}_{\D}$ associated to the vectors $(\mu^k)_{1 \leq k \leq N_T}$ and $(S^k)_{1\leq k \leq N_T}$.
 \end{proposition}
\begin{proof}
Let us first establish estimate \eqref{3.estim_pbdual1}. In this purpose let $1 \leq k \leq N_T$ be fixed. We multiply equation \eqref{3.sch.dual} by $\Delta t(-v^k_{i+1}+2v^k_i-v^k_{i-1})/\Delta x$, we sum over $i \in \I$ and we apply definition \eqref{1.discreteLapl} of the operator $\Delta_\mathcal{T} $ and we obtain
\begin{align*}
 I_{6} + I_{7} = I_{8},
\end{align*}
 with
\begin{align*}
 I_{6} &= \sum_{i \in \I} (v^k_i - v^{k+1}_i) \frac{(-v^k_{i+1}+2v^k_i-v^k_{i-1})}{\Delta x},\\
 I_{7} &= \Delta t \sum_{i \in \I} \mu^k_i \frac{(v^k_{i+1}-2 v^k_i + v^k_{i-1})^2}{\Delta x^3},\\
 I_{8} &= \Delta t \sum_{i \in \I} S^k_i \frac{(-v^k_{i+1}+2v^k_i-v^k_{i-1})}{\Delta x}.
\end{align*}
 For $I_{6}$ reordering the terms leads to
 \begin{align*}
 I_{6} =\frac{1}{\Delta x} \sum_{i \in \I} (v^k_{i+1}-v^k_i) \left[ (v^k_{i+1}-v^k_i) - (v^{k+1}_{i+1}-v^{k+1}_i) \right],
 \end{align*}
 and using the inequality $a(a-b) \geq (a^2-b^2)/2$ we get
 \begin{align}\label{3.I4}
 I_{6} \geq \frac{1}{2} \sum_{i \in \I} \left[ \frac{(v^k_{i+1}-v^k_i)^2}{\Delta x} - \frac{(v^{k+1}_{i+1}-v^{k+1}_i)^2}{\Delta x} \right].
 \end{align}
 For $I_{8}$ applying the Cauchy-Schwarz and Young inequality yield
 \begin{align}\label{3.I6}
 |I_{8}| \leq \frac{\Delta t}{2} \sum_{i \in \I} \Delta x (\mu^k_i)^{-1} |S^k_i|^2 + \frac{\Delta t}{2} \sum_{i \in \I} \mu^k_i \frac{(-v^k_{i+1}+2 v^k_i - v^k_{i-1})^2}{\Delta x^3}.
 \end{align}
 Collecting \eqref{3.I4}--\eqref{3.I6} we obtain
 \begin{align*}
 \sum_{i \in \I} \frac{(v^k_{i+1}-v^k_i)^2}{\Delta x} + \Delta t \sum_{i \in \I} \mu^k_i \frac{(\Delta_\mathcal{T}  v^k)^2_i}{\Delta x} 
 \leq  \sum_{i \in \I} \frac{(v^{k+1}_{i+1}-v^{k+1}_i)^2}{\Delta x} + \Delta t \sum_{i \in \I} \Delta x (\mu^k_i)^{-1} |S^k_i|^2.
 \end{align*}
 In order to prove \eqref{3.estim_pbdual1} it remains to sum over $n \in \lbrace k,\ldots,N_T \rbrace$.

 We now prove estimate \eqref{3.estim_pbdual2}. In this purpose we multiply \eqref{3.sch.dual} by $\Delta x \Delta t$, we sum over $i \in \I$ and $n \in \lbrace k,\ldots,N_T \rbrace$ and we obtain
 \begin{align*}
 \sum_{i \in \I} \Delta x v^k_i = \sum_{n=k}^{N_T} \Delta t \sum_{i \in \I} \mu^k_i \frac{v^n_{i+1}-2 v^n_i + v^n_{i-1}}{\Delta x} + \sum_{n=k}^{N_T} \Delta t \sum_{i \in \I} \Delta x S^n_i.
 \end{align*}
 Applying the Cauchy-Schwarz inequality leads to
 \begin{multline*}
 \left|\sum_{i \in \I} \Delta x v^k_i \right| \leq \left(\sum_{n=k}^{N_T} \Delta t \sum_{i \in \I} \Delta x \mu^n_i \right)^{1/2} \Bigg[ \left(\sum_{n=k}^{N_T} \Delta t \sum_{i \in \I} \mu^n_i \frac{(\Delta_\mathcal{T}  v^n)^2_i}{\Delta x} \right)^{1/2}\\ +  \left(\sum_{n=k}^{N_T} \Delta t \sum_{i \in \I} \Delta x (\mu^n_i)^{-1} |S^n_i|^2 \right)^{1/2} \Bigg].
 \end{multline*}
 Using estimate \eqref{3.estim_pbdual1} we obtain
 \begin{align*}
 \left|\sum_{i \in \I} \Delta x v^k_i \right| \leq 2 \left(\sum_{n=k}^{N_T} \Delta t \sum_{i \in \I} \Delta x \mu^n_i \right)^{1/2} \left(\sum_{n=k}^{N_T} \Delta t \sum_{i \in \I} \Delta x (\mu^n_i)^{-1} |S^n_i|^2 \right)^{1/2}.
 \end{align*}
 Now it remains to apply the discrete Poincar\'e-Wirtinger inequality on the torus obtained in \cite[Lemma 6]{BHR19} in order to conclude the proof of Proposition \ref{prop.estimates_pbdual}.
 \end{proof}
 
 We are now in position to establish the discrete dual estimate.

\begin{theorem}\label{thm.duality}
Let us assume that $\min_{i \in \I} \mu^k_i > 0$ for every $1 \leq k \leq N_T$. Then there exists a constant $C>0$ independent of $\Delta x$ such that the solution $(Z^k)_{1 \leq k \leq N_T}$ to \eqref{1.sch} satisfies
\begin{align*}
\|\mu^{1/2} z \|_{L^2(Q_T)} \leq C \left( 1 + \|\mu\|_{L^1(Q_T)}^{1/2} \right) \, \|z^0\|_{L^2(\Tor)}.
\end{align*}
\end{theorem}
 \begin{proof}
 Let $(v^k_i)_{i \in \I}$ be given in $\R^N$ for every $1\leq k \leq N_T+1$ with $v^{N_T+1}_i=0$ for all $i \in \I$. Now for $1 \leq k \leq N_T$, we multiply \eqref{1.sch} by $\Delta t \Delta x v^k_i$, we sum over $i \in \I$ and $k \in \lbrace 1,\ldots,N_T \rbrace$, we obtain
\begin{align*}
\sum_{k=1}^{N_T} \sum_{i \in \I} \Delta x (z^k_i - z^{k-1}_i) v^k_i - \sum_{k=1}^{N_T} \Delta t \sum_{i \in \I} \Delta x \left( \Delta_\mathcal{T}  (z^k \mu^k)\right)_i v^k_i = 0.
\end{align*}
 Reordering the terms we have
\begin{align}\label{3.dualesti_1}
\sum_{k=1}^{N_T} \Delta t \sum_{i \in \I}  \Delta x z^k_i \left(\frac{(v^k_i - v^{k+1}_i)}{\Delta t} - \mu^k_i \left(\Delta_\mathcal{T}  v^k\right)_i \right) =  \sum_{i \in \I} \Delta x z^0_i v^1_i.
\end{align}
 We define $(S^k_i)_{i \in \I}$ by
\begin{align*}
S^k_i = \frac{(v^k_i - v^{k+1}_i)}{\Delta t} - \mu^k_i \left(\Delta_\mathcal{T}  v^k\right)_i, \quad \forall i \in \I, \, 1\leq k \leq N_T.
\end{align*}
We first notice that $(S^k_i)_{i\in\I}$ is well-defined since we know that equation \eqref{3.sch.dual} is well-posed. Besides applying the Cauchy-Schwarz inequality in \eqref{3.dualesti_1} we get
\begin{align*}
\left|\sum_{k=1}^{N_T} \Delta t \sum_{i \in \I} \Delta x z^k_i S^k_i \right| \leq \|z^0\|_{L^2(\Tor)} \|v^1\|_{L^2(\Tor)}.
\end{align*}
Now, thanks to \eqref{3.estim_pbdual2} we deduce that
\begin{align*}
\left|\sum_{k=1}^{N_T} \Delta t \sum_{i \in \I} \Delta x (\mu^k_i)^{1/2} z^k_i \, (\mu^k_i)^{-1/2} S^k_i \right| \leq C \left(1+\|\mu\|_{L^1(Q_T)}^{1/2} \right) \|\mu^{-1/2} S\|_{L^2(Q_T)} \|z^0\|_{L^2(\Tor)}.
\end{align*}
In the remaining of the proof we want to use the dual definition \eqref{3.defdual.norm} of the norm $\|\cdot\|_{L^2(Q_T)}$. Observe that for any vector $F^k = (f^k_i)_{i \in \I}$, there exists a unique $V = (v_i)_{i \in \I}$ such that
\begin{align*}
(\mu^k_i)^{-1/2} \frac{v_i-v^{k+1}_i}{\Delta t} - (\mu^k_i)^{1/2} \, \left(\Delta_\mathcal{T}  v\right)_i = f^k_i, \quad \forall i \in \I,
\end{align*}
where $(v^{k+1}_i)_{i \in \I}$ is a given vector. Indeed, this system rewrites $(\mathbb{M}^k)^\top V = V^{k+1}+\Delta t \mathbb{D}^k F^k$, for all $1 \leq k \leq N_T$ where $\mathbb{M}^k$ is the invertible tridiagonal matrix given by \eqref{2.defMk} and  $\mathbb{D}^k = \mathrm{diag}((\mu^k_i))_{i \in \I}$. We deduce thanks to formula \eqref{3.defdual.norm}  that it holds
\begin{align*}
\left(\sum_{k=1}^{N_T} \Delta t \sum_{i \in \I} \Delta x \mu^k_i |z^k_i|^2 \right)^{1/2} \leq C \left(1+\|\mu\|_{L^1(Q_T)}^{1/2} \right) \|z^0\|_{L^2(\Tor)}.
\end{align*}
This concludes the proof of Theorem \ref{thm.duality}.
\end{proof}

\subsection{Proof of Theorem \ref{thm.prop}}\label{sec.proof.thm.prop}

We are now able to prove Theorem~\ref{thm.prop}.\medskip

\paragraph{\emph{Step 1: Maximum principle}} Let us first prove the maximum principle satisfies by the solutions to \eqref{5.ic}--\eqref{5.mu2}. Let us notice that for every $1 \leq k \leq N_T$ we have
 \begin{align*}
 \max_{i \in \I} \left|\left(\Delta_\mathcal{T}  \mu^k_1\right)_i \right| &= 
  \max_{i \in \I} \Bigg|d_{11} \sum_{j \in \I} \dx u^k_{1,j} \left(\Delta_\mathcal{T}  \sigma_1\right)_{i-j} +  d_{12} \sum_{j \in \I} \dx u^k_{2,j} \left(\Delta_\mathcal{T}  \rho\right)_{i-j} \Bigg|\\
 &\leq d_{11} \|\Delta_\mathcal{T}  \sigma_1\|_{L^\infty(\Tor)} \sum_{j\in\I} \dx u^k_{1,j} + d_{12} \|\Delta_\mathcal{T}  \rho\|_{L^{\infty}(\Tor)}  \sum_{j \in \I} \Delta x  \,u^k_{2,j}.
 \end{align*}
 Now, let us recall that $m^0_j = \|u^0_j\|_{L^1(\Tor)}$ for $j=1,2$, then thanks to the mass conservation property \eqref{6.conser.mass} we obtain
 \begin{align*}
 \max_{i \in \I} \left|\left(\Delta_\mathcal{T}  \mu^k_1\right)_i \right| &\leq d_{11} m^0_1 \, \|\Delta_\mathcal{T}  \sigma_1\|_{L^{\infty}(\Tor)} +  d_{12} m^0_2\, \|\Delta_\mathcal{T}  \rho\|_{L^{\infty}(\Tor)}.
 \end{align*}
 Similarly we establish the following bound
 \begin{align*}
 \max_{i \in \I} \left|\left(\Delta_\mathcal{T}  \mu^k_2\right)_i \right| \leq d_{22} m^0_2 \, \|\Delta_\mathcal{T}  \sigma_2\|_{L^{\infty}(\Tor)} + d_{21} m^0_1 \, \|\Delta_\mathcal{T}  \rho\|_{L^{\infty}(\Tor)}.
 \end{align*}
 As a direct consequence of the previous estimates and \eqref{2.Linftybounds} (with $\tilde{\gamma}=\gamma$ and $\tilde{\Gamma}=\Gamma$) one obtains point \emph{(i)} of  Theorem~\ref{thm.prop}.\medskip
 
\paragraph{\emph{Step 2: Duality estimate}} Let us now show the discrete duality estimate satisfied by the solutions to \eqref{5.ic}--\eqref{5.mu2}. For every $0 \leq k \leq N_T$ we define the element $z^k_i = u^k_{1,i}+u^k_{2,i}$ for all $i\in \I$. Observe that $z^k_i$ is solution to
\begin{align*}
\frac{z^k_i-z^{k-1}_i}{\Delta t} + \Delta_\mathcal{T} (\mu^k z^k)_i = 0, \quad\text{where}\quad
\mu^k_i = \frac{\mu_{1,i}^k u_{1,i}^k + \mu_{2,i}^k u_{2,i}^k}{u_{1,i}^k+u_{2,i}^k}, \quad \forall i \in \I.
\end{align*}
Thanks to Theorem \ref{thm.SKTexi}, we have $u^k_{1,i}$, $u^k_{2,i} > 0$ for all $i\in\I$ and the element $\mu^k_i$ is well-defined. Besides, applying the discrete duality estimate established in Theorem \ref{thm.duality} we deduce the existence of a constant $C>0$ independent of $\Delta x$ such that
\begin{multline}\label{6.aux.dualesti}
\sum_{k=1}^{N_T} \Delta t \sum_{i \in \I} \Delta x \left(\mu^k_{1,i} u^k_{1,i} + \mu^k_{2,i} u^k_{2,i} \right) \left(u^k_{1,i} + u^k_{2,i} \right) \\ \leq  C \left(1+\sum_{k=1}^{N_T} \Delta t \sum_{i \in \I} \Delta x |\mu^k_i| \right) \left(\sum_{i \in \I} \Delta x |u^0_{1,i}|^2 + \sum_{i \in \I} \Delta x |u^0_{2,i}|^2 \right).
\end{multline}
Now we notice that
\begin{align*}
\sum_{k=1}^{N_T} \Delta t \sum_{i \in \I} \Delta x |\mu^k_i| \leq \sum_{k=1}^{N_T} \Delta t \sum_{i \in \I} \Delta x |\mu^k_{1,i}| + \sum_{k=1}^{N_T} \Delta t \sum_{i \in \I} \Delta x |\mu^k_{2,i}| = I_{9} + I_{10}.
\end{align*}
For $I_{9}$ we have
\begin{align*}
I_{9} &= \sum_{k=1}^{N_T} \Delta t \sum_{i \in \I} \Delta x \left(d_1 + d_{11} \sum_{j \in \I} \dx \sigma_{1,i-j} u^k_{1,j} + d_{12} \sum_{j \in \I} \Delta x \rho_{i-j} u^k_{2,j} \right)\\
&= \sum_{k=1}^{N_T} \Delta t \left( d_1 +d_{11} \sum_{j \in \I} \dx u^k_{1,j} \sum_{i \in \I} \dx \sigma_{1,i-j} + d_{12} \sum_{j \in \I} \Delta x u^k_{2,j} \sum_{i \in \I} \Delta x \rho_{i-j} \right).
\end{align*}
Thus, bearing in mind the mass conservation property \eqref{6.conser.mass} we obtain
\begin{align}\label{6.I11}
I_{9} \leq T(d_1 + d_{11} m^0_1 \, \|\sigma_1\|_{L^1(\Tor)} + d_{12} m^0_2 \, \|\rho\|_{L^1(\Tor)}),
\end{align}
and similarly
\begin{align}\label{6.I12}
I_{10} \leq T(d_2 + d_{22} m^0_2 \, \|\sigma_2\|_{L^1(\Tor)} + d_{21} m^0_1 \, \|\rho\|_{L^1(\Tor)}).
\end{align}
Collecting \eqref{6.aux.dualesti}--\eqref{6.I12} we conclude that point \emph{(ii)} of  Theorem~\ref{thm.prop}  holds.

\section{Convergence of the scheme}\label{sec.conv}

This section is dedicated to the proof of Theorem \ref{thm.convSKT}. In the following the subscript $m$ refer to the size $\eta_m = \max\lbrace \Delta x_m, \Delta t_m \rbrace$ of the family $(\D_m)$ of space-time discretizations of $Q_T$. We derive uniform in $m$ a priori estimates in subsection~\ref{sec.unif} in order to obtain compactness in $L^p(Q_T)$ of the sequences of constant by part reconstructions $(u_{j,m})_m$ for both species $j=1,2$. The compactness results are gathered in Section~\ref{subsec.compactness}. A keypoint is a discrete $L^1$ compactness result obtained in \cite[Lemma 9.2]{ABR11}. This result is the adaptation at the discrete level of a compactness lemma established by Kruzhkov in \cite{Kru69} (see also \cite{Re30}). Finally in Section~\ref{subsec.proof.conv}, we prove Theorem~\ref{thm.convSKT}.

\subsection{Uniform estimates}\label{sec.unif}

In this section we establish some uniform estimates w.r.t. $\Delta x$ and $\Delta t$ fulfilled by the solutions to the scheme \eqref{5.ic}--\eqref{5.mu2}. They rely on the entropy dissipation inequality \eqref{6.entopy.dissip} and the conservation of mass \eqref{6.conser.mass} of Theorem~\ref{thm.SKTexi}.

\begin{proposition}\label{prop.estim}
Let the assumptions of Theorem \ref{thm.SKTexi} hold. Then there exists a constant $C_1>0$ only depending on $d_{12}$, $d_{21}$, $m^0_1$, $m^0_2$  and $H(u^0_1,u^0_2)$ such that
\begin{align}\label{8.unif.BV}
\max_{k=1,\ldots,N_T} \|u^k_j\|_{L^1(\Tor)} + \left(d_j \sum_{k=1}^{N_T} \Delta t \, \|u^k_j\|^2_{1,1,\T} \right)^{\frac{1}{2}} \leq C_1, \quad \mbox{for }j=1,2.
\end{align}
Moreover, assuming that $d_1$ and $d_2$ are positive constants, there exists a constant $C_2 > 0$ only depending on $T$, $d_1$, $d_2$, $d_{11}$, $d_{12}$, $d_{21}$, $d_{22}$, $\|\sigma_1\|_{L^\infty(\Tor)}$, $\|\sigma_2\|_{L^\infty(\Tor)}$, $\|\rho\|_{L^\infty(\Tor)}$, $m^0_1$, $m^0_2$ and $H(u^0_1,u^0_2)$ such that
\begin{align}\label{8.unif.Flux}
\sum_{k=1}^{N_T} \Delta t \sum_{i \in \I} \Delta x \, \left|\F^k_{j,i+\frac{1}{2}} \right| \leq C_2, \quad \mbox{for }j=1,2,
\end{align}
where the numerical fluxes are defined by \eqref{5.sch.flux}.
\end{proposition}
\begin{proof}
The uniform $L^\infty(0,T;L^1(\Tor))$ estimate of the first term in the right hand side of \eqref{8.unif.BV} is a direct consequence of the conservation of mass \eqref{6.conser.mass}. Then, for the uniform discrete $L^2(0,T;W^{1,1}(\Tor))$ estimate, we first notice, for $j=1$ or $2$ and $k \in \lbrace 1,\ldots, N_T \rbrace$, that it holds
\begin{align*}
 |u_j^k|_{1,1,\T} = \sum_{i \in \I} \left|u^k_{j,i+1}-u^k_{j,i}\right| = \sum_{i \in \I} \left|\left(\sqrt{u^k_{j,i+1}}-\sqrt{u^k_{j,i}} \right) \, \left(\sqrt{u^k_{j,i+1}}+\sqrt{u^k_{j,i}} \right) \right|.
\end{align*}
Hence, the Cauchy-Schwarz inequality yields
\begin{align*}
 |u^k_j|_{1,1,\T} \leq \left|\sqrt{u^k_j}\right|_{1,2,\T} \, \left(\sum_{i \in \I} \Delta x \, \left(\sqrt{u^k_{j,i+1}}+\sqrt{u^k_{j,i}} \right)^2 \right)^{\frac{1}{2}}
\end{align*}
Since $(a+b)^2\leq2(a^2+b^2)$ and $\|u^k_j\|_{L^1(\Tor)} = \|u^0_j\|_{L^1(\Tor)}=m^0_j$ (conservation of mass), one has
\begin{align*}
|u^k_j|_{1,1,\T} \leq 2 \left(m^0_j\right)^{1/2} \, \left|\sqrt{u^k_j}\right|_{1,2,\T}.
\end{align*}
Therefore, applying the entropy inequality \eqref{6.entopy.dissip}, we get for the first species
\begin{align*}
d_1\sum_{k=1}^{N_T} \Delta t \, |u^k_1|^2_{1,1,\T} \leq 4 d_1 \, m^0_1\,\sum_{k=1}^{N_T} \Delta t \, \left|\sqrt{u^k_1}\right|^2_{1,2,\T} \leq d_{12} m^0_1 \, H(u^0_1,u^0_2),
\end{align*}
and the equivalent estimate holds for the second species. This yields the existence of $C_1$ such that \eqref{8.unif.BV} holds.

It remains to establish \eqref{8.unif.Flux}. In this purpose we will consider the case $j=1$. Then, using the definition \eqref{5.sch.flux} of the numerical fluxes, we estimate
\begin{align*}
\sum_{k=1}^{N_T} \Delta t \sum_{i \in \I} \Delta x \, \left|\F^k_{1,i+\frac{1}{2}} \right| \leq \sum_{k=1}^{N_T} \Delta t \sum_{i \in \I} \mu^k_{1,i+\frac{1}{2}} \, \left|u^k_{1,i} - u^k_{1,i+1}\right| &+ \sum_{k=1}^{N_T} \Delta t \sum_{i \in \I} u^k_{1,i+\frac{1}{2}} \, \left|\mu^k_{1,i}-\mu^k_{1,i+1} \right|\\
&=  I_{11} + I_{12}.
\end{align*}
For $I_{11}$, applying the regularity of the functions $\sigma_1$ and $\rho$ we have
\begin{align*}
 I_{11} &\leq \sum_{k=1}^{N_T} \Delta t \sum_{i \in \I} \left(d_1 + d_{11} \, \|\sigma_1\|_{L^{\infty}(\Tor)} \, \|u^k_1\|_{L^1(\Tor)} + d_{12} \, \|\rho\|_{L^{\infty}(\Tor)} \, \|u^k_2\|_{L^1(\Tor)} \right) \, \left|u^k_{1,i} - u^k_{1,i+1}\right|.
\end{align*}
Hence, using the conservativity of the scheme and \eqref{8.unif.BV}, we get
\begin{align}\label{8.I11}
 I_{11}\leq \left(d_1 + d_{11} m^0_1 \, \|\sigma_1\|_{L^{\infty}(\Tor)} + d_{12} m^0_2 \, \|\rho\|_{L^{\infty}(\Tor)} \right) \, T^{1/2} \, C_1.
 \end{align}
For $I_{12}$, using the definition of $\mu^k_{1,i}$ for $i \in \I$, we notice that it holds
\begin{multline*}
I_{12} \leq \frac{1}{2} \sum_{k=1}^{N_T} \Delta t \sum_{i \in \I} \Delta x \, \left(u^k_{1,i}+u^k_{1,i+1}\right) \\
\times \left(d_{11} \sum_{n \in \I} \sigma_{1,n} \, \left|u^k_{1,i-n}-u^k_{1,i+1-n}\right| + d_{12} \sum_{n \in \I} \rho_n \, \left|u^k_{2,i-n}-u^k_{2,i+1-n}\right| \right).
\end{multline*}
Then, thanks to the conservativity of the scheme, we obtain
\begin{align*}
I_{12} &\leq 2 m^0_1 \sum_{k=1}^{N_T} \Delta t   \,\left(d_{11} \, \|\sigma_1\|_{L^{\infty}(\Tor)} \, \left|u^k_1\right|_{1,1,\T} + d_{12} \, \|\rho\|_{L^{\infty}(\Tor)} \, \left|u^k_2\right|_{1,1,\T} \right). 
\end{align*}
Therefore, applying the Cauchy-Schwarz inequality and \eqref{8.unif.BV} we end up with
\begin{align}\label{8.I12}
I_{12} \leq 2 \,m^0_1\, T^{1/2} \, C_1 \, \left(\frac{d_{11}}{d_1^{1/2}}  \, \|\sigma_1\|_{L^{\infty}(\Tor)} + \frac{d_{12}}{d_2^{1/2}} \, \|\rho\|_{L^{\infty}(\Tor)} \right).
\end{align}
Collecting \eqref{8.I11} and \eqref{8.I12} and the corresponding inequalities for the second species lead to the existence of $C_2$ such that \eqref{8.unif.Flux} holds. This concludes the proof of Proposition \ref{prop.estim}.
\end{proof}

\subsection{Compactness properties}\label{subsec.compactness}
Let $(u_{1,m},u_{2,m})_{m \in \N}$ be a family, constructed in Theorem \ref{thm.SKTexi}, of finite volume solutions to \eqref{5.ic}--\eqref{5.mu2} associated to the sequence $(\D_m)$. In order to be able to apply \cite[Lemma 9.2]{ABR11}, the first task is to rewrite the scheme \eqref{5.ic}--\eqref{5.mu2} as the discretization of an evolution equation under divergence form. In this purpose we use the equivalent form of \eqref{5.sch} given by \eqref{5.sch.div}. In particular, for $j=1,2$ and $k=1,\ldots,N_T$, we associate to the family of fluxes $\left(\F^k_{j,i+1/2}\right)_{i \in \I}$ the following piecewise reconstruction
\begin{align*}
\F^k_{j,m} = \sum_{i \in \I_m} \F^k_{j,i+\frac{1}{2}} \, \mathbf{1}_{(x_i,x_{i+1})}.
\end{align*}
Then, for this discrete field $\F^k_{j,m}$ we define its $L^1$ norm as
\begin{align*}
\|\F^k_{j,m}\|_{L^1(\Tor)} = \sum_{i \in \I_m} \Delta x \, \left|\F^k_{j,i+\frac{1}{2}} \right|, 
\end{align*}
and its discrete divergence by
\begin{align*}
\mathrm{div}_\T \left(\F^k_{j,m}\right)_i = \frac{1}{\Delta x} \, \left( \F^k_{j,i+\frac{1}{2}} - \F^k_{j,i-\frac{1}{2}} \right), \quad i \in \I_m.
\end{align*}
This definition allows us to rewrite \eqref{5.sch.flux} as
\begin{align}\label{9.schm.div}
\frac{u^k_{j,i}-u^{k-1}_{j,i}}{\Delta t} + \mathrm{div}_\T \left(\F^k_{j,m}\right)_i = 0, \quad \forall i \in \I_m, \, j=1,2,
\end{align}
and we obtain the following result:
 
 \begin{proposition}\label{prop.conv.Lp}
Let the assumptions of Theorem \ref{thm.convSKT} hold and let $(u_{1,m},u_{2,m})_{m\in\N}$ be
a sequence of discrete solutions to \eqref{5.ic}--\eqref{5.mu2} constructed
in Theorem \ref{thm.SKTexi}. Then there exists a subsequence of $(u_{1,m},u_{2,m})$, which is not relabeled, and $(u_1,u_2)\in \left(L^p(Q_T)\right)^2$, with $p \in [1,3)$, such that
\begin{align*}
u_{j,m} \rightarrow u_j \quad \mbox{strongly in }L^p(Q_T) \mbox{ for }1 \leq p < 3, \quad \mbox{as }m \to \infty,
\end{align*}
and almost everywhere.
\end{proposition}
\begin{proof}
A direct consequence of Proposition \ref{prop.estim} is that there is a constant $C$ independent of $\Delta x_m$ and $\Delta t_m$ such that
\begin{align*}
\sum_{k=1}^{N_T} \Delta t_m \, \|u^k_{j,m}\|_{L^1(\Tor)} + \sum_{k=1}^{N_T}  \Delta t_m \, \|\F^k_{j,m}\|_{L^1(\Tor)} + \sum_{k=1}^{N_T} \Delta t_m \, |u^k_{j,m}|_{1,1,\T_m} \leq C, \quad j=1,2.
\end{align*}
By \cite[Lemma 9.2]{ABR11}, which can be applied thanks to \eqref{9.schm.div}, there is a function $u_j \in L^1(Q_T)$, $j=1,2$, such that, up to a subsequence,
\begin{align*}
u_{j,m} \rightarrow u_j \quad \mbox{strongly in }L^1(Q_T) \mbox{ as }m \to \infty.
\end{align*}
Moreover, Proposition \ref{prop.estim} also implies that the sequence $(u_{i,m})$ is uniformly bounded in the space $L^\infty(0,T;L^1(\Tor))$ and in $L^2(0,T;BV(\Tor))$. The continuous embedding of $BV(\Tor)$ in $L^\infty(\Tor)$ (see \cite{AFP00}) implies that the sequence $(u_{j,m})$ is uniformly bounded in $L^2(0,T;L^\infty(\Tor))$. Hence, by interpolation, one has a uniform bound of $(u_{j,m})$ in $L^3(Q_T)$. Thus, Vitali's theorem gives the strong convergence of $(u_{j,m})$ towards $u_j$ in $L^p(Q_T)$ for all $p \in [1,3)$. This concludes the proof of Proposition \ref{prop.conv.Lp}.
\end{proof}

\begin{remark}\label{rem.dimsupconv}
In dimension $d\geq2$, we have the compact embedding of the space $BV(\Tor^d)$ in $L^{\frac{d}{d-1}}(\Tor^d)$. In particular in this case the sequence $(u_{j,m})$ is uniformly bounded in $L^{\frac{d}{d-1}}(Q_T)$. Therefore arguing as in the previous proof we deduce the existence for $j=1$ and $2$ of $u_j \in L^p(Q_T)$ for $p \in [1,{d}/{(d-1)})$, such that, up to a subsequence,
\begin{align*}
u_{j,m} \rightarrow u_j \quad \mbox{strongly in }L^p(Q_T) \mbox{ for }1 \leq p < {\frac{d}{d-1}}, \quad \mbox{as }m \to \infty.
\end{align*}
\end{remark}

\begin{corollary}\label{prop.conv.mu.LpLinf}
Let the assumptions of Proposition \ref{prop.conv.Lp} hold. Then there exists a subsequence of $(u_{1,m},u_{2,m})$, such that for any $p \in [1,3)$ one has 
\begin{equation*}
\begin{aligned}
\mu_{1,m} \rightarrow \mu_1 = d_1+d_{12}\rho_1\ast u_2 +d_{11}\sigma_1\ast u_1 \quad \mbox{strongly in }L^p(0,T; L^\infty(\Tor)), \quad \mbox{as }m\to \infty,\\
\mu_{2,m} \rightarrow \mu_2 = d_2+d_{21}\rho_2\ast u_1 +d_{22}\sigma_2\ast u_2  \quad \mbox{strongly in }L^p(0,T; L^\infty(\Tor)), \quad \mbox{as }m\to \infty,
\end{aligned}
\end{equation*}
where we recall that $\rho_1(x)=\rho_2(-x)=\rho(x)$ for a.e. $x \in \Tor$.
\end{corollary}

\begin{proof}
Observe that by definition \eqref{5.mu1}-\eqref{5.approx.conv1} one has for $x \in K_i$ and $t\in(0,T)$
\[
\begin{aligned}
\mu_1(x,t) - \mu_{1,m}(x,t) &= d_{12} \, (\rho_1\ast (u_2-u_{2,m}))(x,t) +d_{11} \, (\sigma_1\ast (u_1-u_{1,m}))(x,t)\\
&+ d_{12}\int_{\Tor}(\rho_1(x-y) - \rho_1(x_i-y)) \, u_{2,m}(y,t) \,\mathrm{d}y\\
&+ d_{11}\int_{\Tor}(\sigma_1(x-y) - \sigma_1(x_i-y)) \, u_{1,m}(y,t)\,\mathrm{d}y.\\
\end{aligned}
\]
Therefore by dominated convergence one obtains the strong convergence of $\mu_{1,m}$ towards $\mu_1$ in $L^1(Q_T)$ and almost everywhere in $Q_T$. Besides, for a.e. $t\in(0,T)$ thanks to Young's inequality and for $p\in(1,3)$, it holds
\begin{align*}
\|\rho_1 \ast (u_2-u_{2,m})(t)\|_{L^\infty(\Tor)} \leq \|\rho_1\|_{L^{\frac{p}{p-1}}(\Tor)} \, \|(u_2-u_{2,m})(t)\|_{L^p(\Tor)}.
\end{align*}
Then, applying Proposition \ref{prop.conv.Lp}, we obtain
\begin{align*}
\int_0^T \|\rho_1 \ast (u_2-u_{2,m})(t)\|^p_{L^\infty(\Tor)} \mathrm{d}t \leq \|\rho_1\|^p_{L^{\frac{p}{p-1}}(\Tor)} \, \int_0^T \|(u_2-u_{2,m})(t)\|^p_{L^p(\Tor)} \, \mathrm{d}t \to 0 \quad \mbox{as }m \to \infty.
\end{align*}
Let us now setting $\xi(x,y) = \rho_1(x-y) - \rho_1(x_i-y)$ for a.e. $x\in K_i$ and $y\in\Tor$. Hence, for a.e. $t \in (0,T)$, the H\"older inequality yields for $p\in(1,3)$
\[
\left\|\int_{\Tor}\xi(\cdot,y)u_{2,m}(t,y)\mathrm{d}y\right\|_{L^\infty(\Tor)} \leq \sup_{|z|\leq\dx_m}\|\rho_1(z+\cdot) - \rho_1\|_{L^{\frac{p}{p-1}}(\Tor)}\|u_{2,m}(t)\|_{L^p(\Tor)}.
\]
 The first factor in the right hand side tends to $0$ (by density of continuous functions in $L^{p/(p-1)}(\Tor)$) while, bearing in mind Proposition \ref{prop.conv.Lp}, the second factor is uniformly bounded in $L^p(Q_T)$. Therefore one can conclude the strong convergence in $L^p(0,T;L^\infty(\Tor))$ by using Young's inequality and the previous argument. This finishes the proof of Corollary \ref{prop.conv.mu.LpLinf}.
\end{proof}

\subsection{Proof of Theorem \ref{thm.convSKT}}\label{subsec.proof.conv}

It remains to prove that the functions $u_1$ and $u_2$ constructed in Section \ref{subsec.compactness} are distributional solutions to \eqref{4.SKT1}--\eqref{4.ic} in the sense of \eqref{weak1}. Let $\phi \in C^\infty_0(\Tor \times [0,T))$, we multiply equation \eqref{5.sch} by $\Delta t_m \dx_m \, \phi^{k-1}_i$, where $\phi^{k-1}_i = \phi(x_i, t^{k-1})$, and we sum over $i \in \I$ and $k \in \lbrace 1,\ldots,N_T \rbrace$. We obtain $
F^m_1 + F^m_2 = 0$ with
\begin{align*}
F^m_1 &= \sum_{k=1}^{N_T} \sum_{i \in \I} \dx_m \, (u^k_{1,i}-u^{k-1}_{1,i}) \phi^{k-1}_i= \sum_{k=1}^{N_T} \sum_{i \in \I} \dx_m \, u^{k}_{1,i}(\phi^{k-1}_i - \phi^{k}_i) - \sum_{i \in \I} \dx_m u^{0}_{1,i} \phi^{0}_{i},
\end{align*}
and
\begin{align*}
F^m_2 &= \sum_{k=1}^{N_T} \Delta t_m \sum_{i \in \I} \frac{-\mu^k_{1,i+1} u^k_{1,i+1} + 2 \mu^k_{1,i} u^k_{1,i} - \mu^k_{1,i-1} u^k_{1,i-1}}{\dx_m} \, \phi^{k-1}_i\\
&= \sum_{k=1}^{N_T} \Delta t_m \sum_{i \in \I} \frac{-\phi^{k-1}_{i+1} + 2 \phi^{k-1}_{i} - \phi^{k-1}_{i-1}}{\dx_m} \, \mu^k_{1,i} u^k_{1,i}.
\end{align*}
Let $\psi_m(x,t) = (\phi^{k}_i - \phi^{k-1}_i)/\Delta t_m$ for all $x\in K_i$ and $t\in(t_{k-1}, t_k]$ and $\phi_m^0(x) = \phi^{0}_i$ for all $x\in K_i$. Then, since $\psi_m \to\partial_t\phi $ in $L^\infty(Q_T)$ and $\phi^0_m\to \phi^0$ in $L^\infty(\Tor)$, so using the convergence results of Proposition~\ref{prop.conv.Lp} one obtains
\begin{align*}
&F^m_1 +\int_{Q_T} u_1 \partial_t \phi \, \mathrm{d}x \mathrm{d}t + \int_{\Tor} u^0_1(x) \phi(x,0) \, \mathrm{d}x\\
=& \int_{Q_T} (u_1 \partial_t \phi - u_{1,m}\psi_m)\, \mathrm{d}x \mathrm{d}t +  \int_{\Tor} (u^0_1 \phi(\cdot,0) - u^0_1\phi_m^0)\, \mathrm{d}x \to0\quad \mbox{as }m\to 0.
\end{align*}
Similarly, if one defines $\zeta_m(x,t)= (\phi^{k-1}_{i+1} - 2 \phi^{k-1}_{i} + \phi^{k-1}_{i-1})/(\dx_m)^2$ for all $x\in K_i$ and $t\in(t_{k-1}, t_k]$, then $\zeta_m\to \Delta \phi $ in $L^\infty(Q_T)$ and therefore using Proposition~\ref{prop.conv.Lp} and Corollary~\ref{prop.conv.mu.LpLinf} one obtains
\begin{align*}
&F^m_2 +\int_{Q_T} u_1\mu_1 \Delta\phi \, \mathrm{d}x \mathrm{d}t= \int_{Q_T} (u_1 \mu_1\Delta\phi - u_{1,m}\mu_{1,m}\zeta_m)\, \mathrm{d}x \mathrm{d}t\to0\quad \mbox{as }m\to 0.
\end{align*}
This concludes the proof of Theorem \ref{thm.convSKT}.

\section{Numerical experiments}\label{sec.num}
In this section, we perform several numerical experiments to illustrate the behavior of the scheme.

\subsection{Implementation}

The scheme was implemented in dimension $d=1$ and $d=2$ using Matlab. The code is available at \url{https://gitlab.inria.fr/herda/nonlocal-skt}. In order to optimize the computational cost, a number of matrices can be pre-assembled and stored using a sparse matrix structure. This is the case for the matrix of the Laplacian and those related to the convolution kernels. Moreover, the assembling can be performed efficiently using the discrete Fourier transform. At each time step the nonlinear system is solved using a Newton method. Convergence of the Newton method is reached when the $\ell^\infty$ norm of the residue divided by the norm of the first guess gets less than a given tolerance, which we took to be $10^{-10}$ in our experiments. An adaptive time step procedure is implemented in case the Newton method fails to converge. After maximum number of steps ($50$ in the experiments), if the target error is not attained, $\Delta t$ is divided by $2$. If there was refinement on a given time step,  $\Delta t$ is multiplied by two for the next time step. In the experiments below the Newton method never failed to converge and the time step remained constant along all the simulations.

\subsection{Test case 1: Convergence for various convolution kernels and initial data}\label{sec:testcase1}
In this first test case, we investigate the convergence of the scheme in the case for the following nonlocal cross-diffusion system
\begin{align*}
\partial_t u_1 - \partial_{xx}^2((\rho \ast u_2) u_1) &= 0,\\
\partial_t u_2 - 2\partial_{xx}^2((\rho \ast u_1) u_2) &= 0.
\end{align*}
The convolution kernel is taken to be either the Dirac measure, which we denote by $\rho_0$, either by an approximation of a Dirac
\begin{equation}\label{eq:rhodelt}
\rho_\delta(x) =\delta^{-1} \chi_{[-\delta/2,\delta/2]}(x),
\end{equation}
where $\chi_A$ indicator function of the set $A$, or the smooth kernel
\[
\rho_\text{smooth}(x) = \cos(\nu_Lx)+1.
\]
with $\nu_L = 2\pi/L$. We consider two initial data, either the indicator functions
\begin{equation}\label{eq:CI1}
u^0_1(x) = \chi_{[\frac L9, \frac L3]}(x)\,,\quad u^0_2(x) = \chi_{[\frac L3, \frac{3L}{4}]}(x)\,,\quad x\in\mathbb{R}/L\mathbb{Z}.
\end{equation}
or the smooth functions
\begin{equation}\label{eq:CI2}
u^0_1(x) = \cos\left(\nu_L x\right)+1\,,\quad u^0_2(x) = \sin\left(\nu_L x\right)+1\,,\quad x\in\mathbb{R}/L\mathbb{Z}.
\end{equation}
The final time of simulation is taken to be $T = 5$ and the domain has length $L=25$. We run the scheme for a sequence of decreasing space and time steps. More precisely the number of points is $N_k = 32\cdot2^{k-1}$ for $k=1,\dots,6$ and the corresponding time step $\Delta t_k = \Delta t_0\cdot4^{-(k-1)}$, with $\Delta t_0 = 5$. Observe that the refinement of the time step allows to witness experimental convergence in space up to second order accuracy if it is attained. As we do not know the analytical solution for this system, we take as reference solution the computed solution on the finest mesh ($N=1024$). Then the error for the $k$-th mesh is taken to be the $\ell^\infty$ norm between the $k$-th solution and the reference solution projected on the $k$-th mesh. From these errors the experimental order is evaluated by linear regression (in log scale). In Table~\ref{tab:convergence}, we report the experimental order of convergence and the error between the $N=512$ mesh and $N=1024$ mesh for each kernel and initial data.
\begin{table}[!h]
\begin{tabular}{l|r|r|r}
Convolution kernel:&Smooth $\rho_{\text{smooth}}$&Indicator $\rho_{L/4}$&Dirac $\rho_0$\\\hline\hline
Initial condition: &order: $1.97$&order: $1.53$ &order: $1.04$\\
indicator func. \eqref{eq:CI1}&error: $5\cdot10^{-3}$&error: $7\cdot10^{-2}$&error: $2.7\cdot10^{-3}$\\\hline
Initial condition: &order: $2.32$&order: $ 2.02$ &order: $2.32$\\
smooth func.  \eqref{eq:CI2}&error: $4.9\cdot10^{-4}$&error: $9.2\cdot10^{-4}$&error: $5\cdot10^{-4}$
\end{tabular}
\caption{Estimated order of convergence in space and absolute error at final time for the mesh $N=512$ in $L^\infty$ norm for various convolution kernels and initial data. Reference solution is for $N=1024$.}
\label{tab:convergence}
\end{table}

\subsection{Test case 2: From nonlocal to local cross-diffusion}\label{sec:testcase2} As a second test case, we investigate numerically the rate of convergence for different metrics of the so-called localization limit. Namely we study the rate of convergence of solutions of the nonlocal cross-diffusion system \eqref{4.SKT1}--\eqref{4.SKT2} towards solutions of its corresponding local version \eqref{4.local.SKT1}--\eqref{4.local.SKT2} as the convolution kernel tends to a Dirac measure. Indeed, if theoretically this localization limit has been proved in \cite{DM21,AM20}, the proofs rely on some compactness method and no explicit ``error'' bounds  are available (see also for instance \cite{JPZ22}). The establishment of such explicit estimates seems to be a complex task. In order to get a better understanding of this problem we aim to study this question thanks to our finite volume scheme.

More precisely, in this test case we consider the same system as in the first test case with $\rho = \rho_{\delta}$ for various values of $\delta\in[0,L]$. The domain has length $L=25$, the final time is $T=1$ and the mesh is such that $N = 1024$ and $\Delta t = 10^{-2}$. We evaluate the error at time $T$ between the solution $(u_1^{(\delta)}, u_2^{(\delta)})$ computed for the kernel $\rho_\delta$ and $(u_1^{(0)}, u_2^{(0)})$ computed for the local cross-diffusion system in Wasserstein-1 norm,
\[
W_1(u_1^{(\delta)}, u_1^{(0)}) + W_1(u_2^{(\delta)}, u_2^{(0)}),
\]
and in $L^p$ norms 
\[
\|u_1^{(\delta)}- u_1^{(0)}\|_{L^p} + \|u_2^{(\delta)}- u_2^{(0)}\|_{L^p},
\]
with $p=1$ or $p=\infty$. For the computation of the Wasserstein-1 norm we recall that in dimension $1$, if $f$ and $g$ are non-negative integrable functions on $\mathbb{R}$ with the same mass, one has $W_1(f, g) = \|F-G\|_{L^1}$, with $F'=f$, $G' =g$ and $F(-\infty) = G(-\infty)$. In practice, $f$ and $g$ are piecewise constant functions, thus the previous norm can be computed exactly numerically. On Figure~\ref{fig:localnonlocal}, we plot the error as a function of $\delta/L=4W_1(\rho_\delta, \rho_0)/L$ for the two initial data \eqref{eq:CI1} and \eqref{eq:CI2}.
\begin{figure}[h!]
  \centering
    \begin{tabular}{lcr}
    \begin{minipage}{.35\textwidth}
%
%
%
\begin{tikzpicture}

\begin{axis}[%
width=\textwidth,
scale only axis,
xmode=log,
xmin=0.001,
xmax=2,
xminorticks=true,
xlabel={$\delta/L$},
ymode=log,
ymin=0.01,
ymax=10,
yminorticks=true,
axis background/.style={fill=white},
legend style={at = {(1.0,0.0)}, anchor=south east, legend cell align=left, align=left, draw=white!15!black}
]
\addplot [color=red, dashed, mark=triangle, mark options={solid, red}]
  table[row sep=crcr]{%
1	3.23002740349045\\
0.5	2.54114578626428\\
0.4	2.13556278264275\\
0.3	1.62994743369678\\
0.2	1.08205256734307\\
0.1	0.597556521931737\\
0.05	0.427735833421858\\
0.04	0.382083711463199\\
0.03	0.322736374386223\\
0.02	0.248706983214396\\
0.01	0.164064432233567\\
0.005	0.102057279805887\\
0.003	0.0745733610728946\\
};
\addlegendentry{$W_1$ (order: $0.65$)}

\addplot [color=blue, dashed, mark=square, mark options={solid, blue}]
  table[row sep=crcr]{%
1	4.05319804246771\\
0.5	3.0502225890075\\
0.4	2.7534483929286\\
0.3	2.36650877805009\\
0.2	1.95331793256302\\
0.1	1.55330288785731\\
0.05	1.33694907023247\\
0.04	1.24910183674991\\
0.03	1.11586933285021\\
0.02	0.915577848065252\\
0.01	0.6911529631725\\
0.005	0.483569480376013\\
0.003	0.375539748971271\\
};
\addlegendentry{$L^1$ (order: $0.38$)}

\addplot [color=black, dashed, mark=diamond, mark options={solid, black}]
  table[row sep=crcr]{%
1	1.0922154014351\\
0.5	0.974826583203721\\
0.4	1.0399393540034\\
0.3	1.06483925440273\\
0.2	1.05138527105017\\
0.1	1.03923396554592\\
0.05	1.02940148982679\\
0.04	1.04755977464866\\
0.03	1.07040103139088\\
0.02	1.09293306452348\\
0.01	1.10725731750609\\
0.005	1.10376632763749\\
0.003	1.08764248877969\\
};
\addlegendentry{$L^\infty$}

\end{axis}
\end{tikzpicture}%
    \end{minipage}
    &\begin{minipage}{.03\textwidth}~\end{minipage}&
    \begin{minipage}{.35\textwidth}
%
%

%
\begin{tikzpicture}

\begin{axis}[%
width=\textwidth,
scale only axis,
xmode=log,
xmin=0.001,
xmax=2,
xminorticks=true,
xlabel={$\delta/L$},
ymode=log,
ymin=1e-6,
ymax=10,
yminorticks=true,
axis background/.style={fill=white},
legend style={at = {(1.0,0.0)}, anchor=south east, legend cell align=left, align=left, draw=white!15!black}
]
\addplot [color=red, dashed, mark=triangle, mark options={solid, red}]
  table[row sep=crcr]{%
1	17.36869633489\\
0.5	5.01782434611946\\
0.4	3.12438286330904\\
0.3	1.72643288535814\\
0.2	0.756032123487418\\
0.1	0.188157131993878\\
0.05	0.0459226722472736\\
0.04	0.0296581616728354\\
0.03	0.0169408069271463\\
0.02	0.00776230345867207\\
0.01	0.00211661549991365\\
0.005	0.000423300471877955\\
0.003	0.00014109889930773\\
};
\addlegendentry{$W_1$ (order: $2.01$)}

\addplot [color=blue, dashed, mark=square, mark options={solid, blue}]
  table[row sep=crcr]{%
1	4.53106620994319\\
0.5	1.46211826081622\\
0.4	0.940539190241608\\
0.3	0.536690541887217\\
0.2	0.243283675512452\\
0.1	0.0625753631515012\\
0.05	0.0154540384762746\\
0.04	0.00999614764469887\\
0.03	0.00571684168992667\\
0.02	0.00262190103930717\\
0.01	0.000715343110626223\\
0.005	0.000143085165811804\\
0.003	4.7695969690004e-05\\
};
\addlegendentry{$L^1$ (order: $1.98$)}

\addplot [color=black, dashed, mark=diamond, mark options={solid, black}]
  table[row sep=crcr]{%
1	0.394312773801331\\
0.5	0.128263593682307\\
0.4	0.0818146872499927\\
0.3	0.0460837671720641\\
0.2	0.0207525952941274\\
0.1	0.005374429973556\\
0.05	0.00138050133490772\\
0.04	0.000900583413454739\\
0.03	0.000518892119064224\\
0.02	0.000239393561033516\\
0.01	6.55729413716472e-05\\
0.005	1.31322153718999e-05\\
0.003	4.37839287863173e-06\\
};
\addlegendentry{$L^\infty$ (order: $1.97$)}

\end{axis}
\end{tikzpicture}%
    \end{minipage}
    \end{tabular}
\caption{Distance between solution of the nonlocal and local cross-diffusion system at final time versus $\delta/L$. Left: Initial data is indicator function \eqref{eq:CI1}; Right: Initial data is the smooth function \eqref{eq:CI2}}.
\label{fig:localnonlocal}
\end{figure}
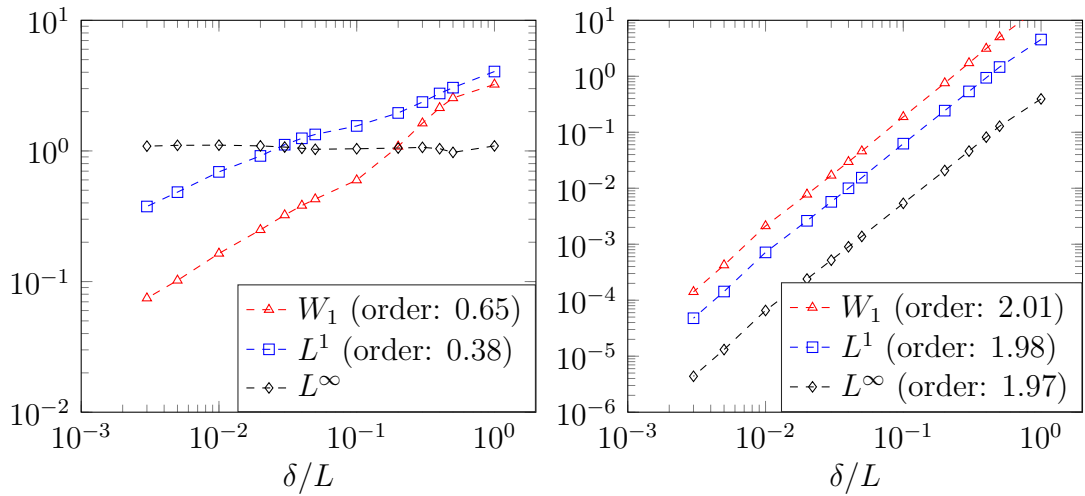
 For the smooth initial data \eqref{eq:CI2} supported on the whole domain (up to one point), we observe convergence with rate $O(W_1(\rho_\delta, \rho_0)^2)$ for all the norms. For the discontinuous initial data \eqref{eq:CI2} supported on part of the domain, there is no experimental convergence in $L^\infty$ norm, and $O(W_1(\rho_\delta, \rho_0)^\alpha)$ convergence with $\alpha \approx 0.38$ in $L^1$ norm and $\alpha = 0.65$ in Wasserstein-1 norm.
 
\subsection{Test case 3: Turing instabilities in prey-predator systems with nonlocal cross-diffusion}\label{sec:testcase3}
In this last test case, we consider the following system with nonlocal cross diffusion and reaction modelling a population of preys with density $u_1$ and predators with density $u_2$. The system reads
\begin{align*}
\partial_t u_1 - d_1\Delta u_1 &= R_1(u_1,u_2),\\
\partial_t u_2 - \Delta((d_2 + d_{21} \, \rho_2 \ast u_1) u_2) &= R_2(u_1,u_2).
\end{align*}
 The precise reaction terms  will be specified below. On the one hand, preys are subject to linear diffusion with constant diffusivity coefficient $d_1$. However, the predators diffuse depending on the presence or the absence of preys. More precisely, the convolution kernel $\rho_2$ is chosen such that it is close to $0$ near the origin and large away form the origin  (up to a given distance). This models the fact that predators need not seek for preys when they are available at their position, while they shall diffuse more rapidly if higher densities of preys are ahead. The reaction terms will be chosen following the phytoplankton-herbivore model of Segel and Levin \cite{levin1976hypothesis} and a variation of Mimura-Nishiura-Yamaguti \cite{mimura_1979}. In both cases, the particularities are an  autocatalytic effect on the phytoplankton's (preys) growth rate and a density-dependent mortality of herbivore (predators). In the case of linear diffusion, this model is famous for exhibiting diffusive instabilities \cite{levin1976hypothesis} around the homogenenous equilibrium. The corresponding Turing patterns have been invoked to justify the patchiness of phytoplankton's distribution in the oceans \cite{levin1976hypothesis}.  In \cite{levin1976hypothesis} Segel and Levin mention that in these models \emph{the assumption of passive diffusion is made for simplicity only; more complicated movement patterns can also lead to diffusive instability.} Here we propose a more complex description model of the behavior of predators thanks to non-local cross-diffusion. In the following, we illustrate numerically the persistence and the modification of Turing patterns in the presence of nonlocal cross-diffusion. 
\subsubsection{One dimensional case: Segel-Levin reaction term} We consider the one-dimensional case with the following reaction terms 
\begin{align*}
R_1(u_1,u_2)= au_1+eu_1^2-bu_1u_2,\quad R_2(u_1,u_2)= -du_2^2+cu_1u_2.
\end{align*}
where the parameters are $a=b=c=d=1$ and $e=\frac{1}{3}$. Concerning the diffusion we consider two cases. In the first case, both species are driven by linear diffusion with $d_1 = 0.05$ for preys and $d_2 = 2$ and without cross-diffusion $d_{21}=0$. In the second case the preys are driven by linear diffusion with $d_1 = 0.05$ and the predators by nonlocal cross-diffusion with $d_{21} = 1$ and the kernel
\[
\rho_2(x) = C_r[x^2\chi_{(-r,r)}(x) + (x-2r)^2\chi_{[r,2r)}(x) + (x+2r)^2\chi_{(-2r,-r]}(x)],
\]
with $C_r$ a normalizing constant such that $\int\rho_2 = 1$. This kernel vanishes at $0$, has support on $[-2r,2r]$ and is maximal at $x = \pm r$. It is designed to model the hunting behavior of predators which will diffuse if most of the preys are away from their position, with a detection radius equal to $r$ and a maximal distance of detection of $2r$. In both cases the simulation is performed on a domain of length $L = 25$ with $N = 500$ cells. The radius is taken to be  $r = 10L/49$. The final time of simulation $T=500$ and the time step is $\Delta t =0.1$. The initial data is taken as a small perturbation of the homogeneous equilibrium
\[
u^0_1(x) = \frac{ad}{bc-de} + \varepsilon\chi_{[L/3,L/9]}(x)\,,\quad u^0_2(x) = \frac{ac}{bc-de},
\]
with $\varepsilon = 10^{-2}$. With the chosen parameters, the homogeneous equilibrium is linearly unstable in both the linear diffusion and the nonlocal cross-diffusion cases. Numerically we observe the solution converges in time towards an heterogeneous equilibrium in both cases. On Figure~\ref{fig:turing1d}, we plot the densities of preys and predators at final time. The difference between the patterns in the two cases is illustrated.
\begin{figure}[!h]
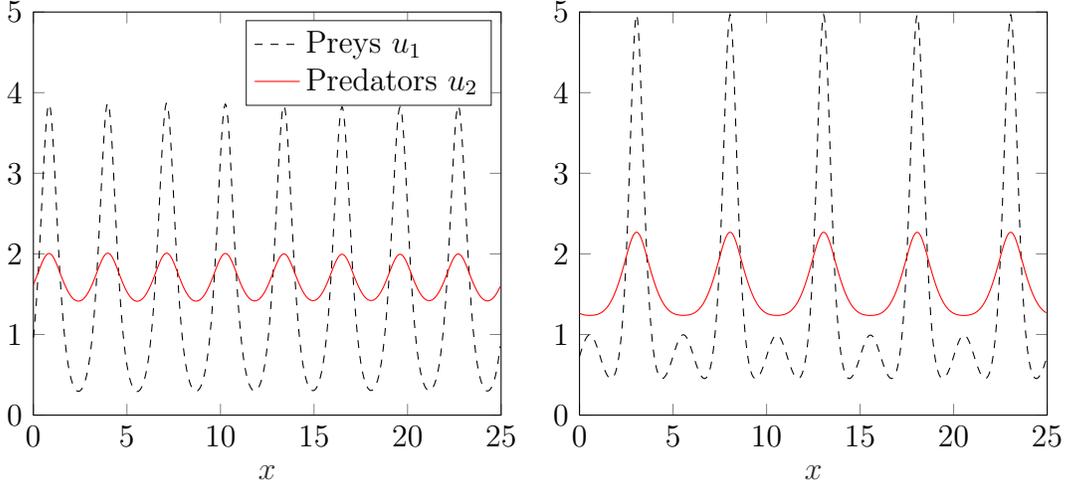

    \centering
    \begin{tabular}{lr}
    \begin{minipage}{.4\textwidth}
    \include{turing_lindiff_1d}
    \end{minipage}
    &
    \begin{minipage}{.4\textwidth}
    \include{turing_crossdiff_1d}
    \end{minipage}
    \end{tabular}
    \caption{Turing patterns at final time for (left) linear diffusion for predators and preys ($d_1 = 0.05$, $d_2 = 2$ and $d_{21} = 0$) and  (right) cross-diffusion for predators and linear diffusion for preys ($d_1 = 0.05$, $d_2 = 0$ and $d_{21} = 1$).}
    \label{fig:turing1d}
\end{figure}

\subsubsection{Two dimensional case: Mimura-Nishiura-Yamaguti reaction term}
Now we consider the two-dimensional case with the following reaction terms 
\begin{align*}
R_1(u_1,u_2)= au_1+eu_1^2-du_1^3-bu_1u_2, \quad R_2(u_1,u_2)= -fu_2-gu_2^2+cu_1u_2,
\end{align*}
with $a=35/9$, $b=c=f=1$, $d=1/9$, $e=16/9$ and $g=2/5$. For these reaction terms, the homogeneous equilibrium is given for the preys by 
\begin{align*}
[f + (c(b^2c^2 - 2bceg + 4dfbg + e^2g^2 + 4adg^2)^{1/2} - bc^2 + ceg - 2dfg)/(2dg)]/c = 5,
\end{align*}
and by 
\begin{align*}
    [c(b^2c^2 - 2bceg + 4dfbg + e^2g^2 + 4adg^2)^{1/2} - bc^2 + ceg - 2dfg]/(2dg^2) = 10,
\end{align*}    
for predators. Concerning the diffusion we consider three cases. In the first case, both species are driven by linear diffusion with $d_1 = 0.001$ for preys and $d_2 = 4$ and without cross-diffusion $d_{21}=0$. Similarly to the one dimensional test case, in the second case the preys are driven by linear diffusion with $d_1 = 0.001$ and the predators by nonlocal cross-diffusion with $d_{21} = 2/5$ and the kernel is the indicator function of an annulus 
\[
\rho_2^\text{sym}(x,y) = C\chi_{(3/8,1/2)}(x^2+y^2),
\]
with $C$ a normalizing constant such that $\int\rho_2^\text{sym} = 1$.  The third case is the same has the second case with linear diffusion for preys and nonlocal cross-diffusion for predators but the kernel is not symmetric and given by
\[
\rho_2^\text{nonsym}(x,y) = C\chi_{(3/8,1/2)}(x^2+y^2)\chi_{[0,\infty)}(x)\chi_{[0,\infty)}(y),
\]
with $C$ a normalizing constant such that $\int\rho_2^\text{nonsym} = 1$. In terms of modelling, it means that predators only sense preys that are north-east of their position (upper right quadrant). The final time of simulation $T=20$ and the time step is $\Delta t =0.01$. In any cases the simulation is performed on a domain of horizontal length $L_x = 4$ and vertical length $L_y = 3$ with $133\times100$ cells.  The initial data is taken as a small perturbation of the homogeneous equilibrium
\[
u^0_1(x,y) = 5 + \varepsilon\chi_{[L_x/9, 4L_x/9]\times[7L_y/9, 8L_y/9]}(x,y)\,,\quad u^0_2(x,y) = 10,
\]
and $\varepsilon = 10^{-2}$. Once again with the chosen parameters, the homogeneous equilibrium is linearly unstable in all cases and the solution converges in time towards an heterogeneous equilibrium. On Figure~\ref{fig:turing2d}, we plot the colormap density of preys at final time. The difference between the patterns in the three cases is illustrated. In the last case the patterns are consistent with the breaking of symmetry in the kernel $\rho_2^\text{nonsym}$.
\begin{figure}[!h]
    \centering
    \begin{tabular}{c}
    \begin{minipage}{\textwidth}
    \begin{tabular}{lr}
       \includegraphics[width = .45\textwidth]{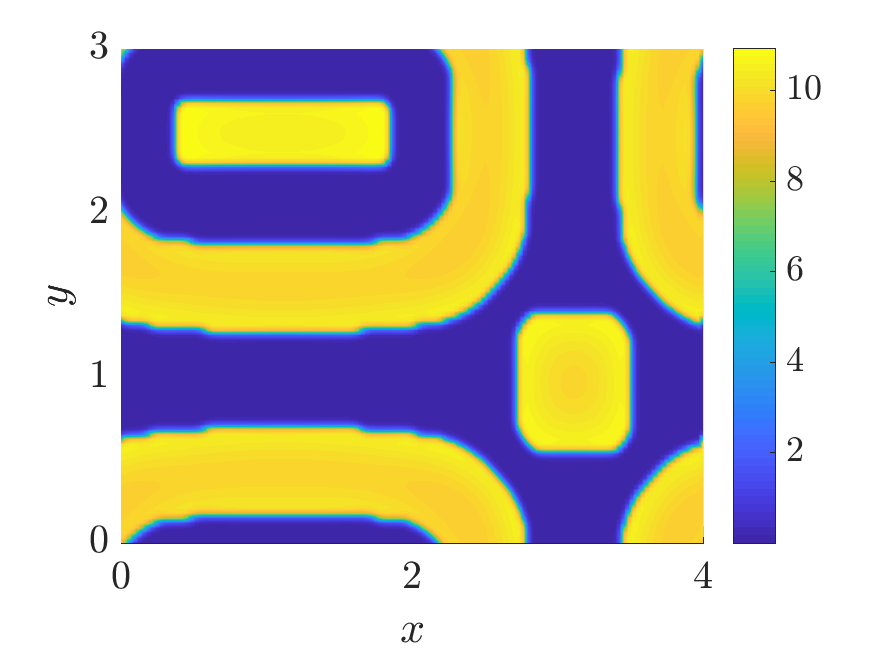}&
   \includegraphics[width = .45\textwidth]{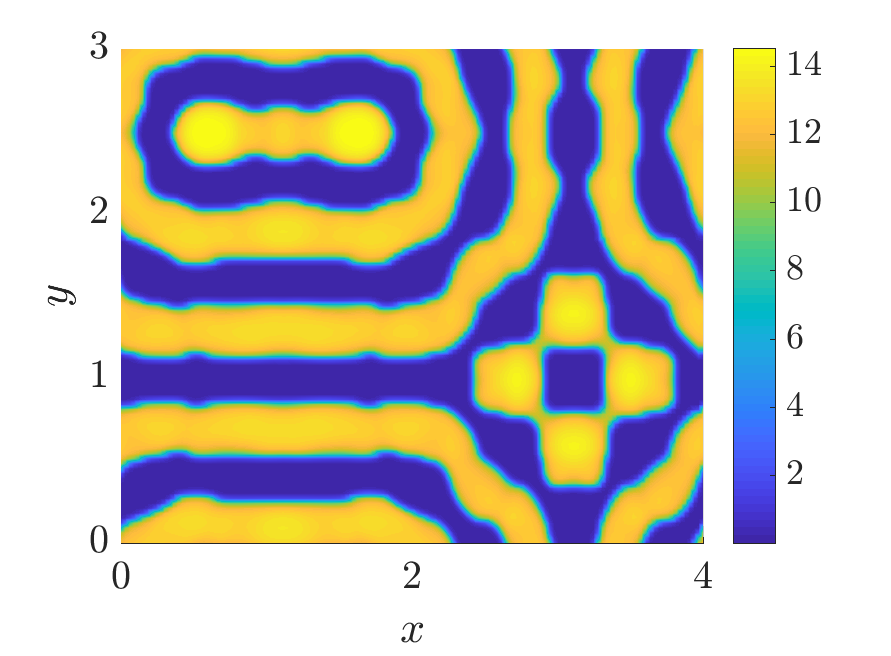}
   \end{tabular}
    \end{minipage}\\
    \begin{minipage}{\textwidth}
    \centering
   \includegraphics[width = .45\textwidth]{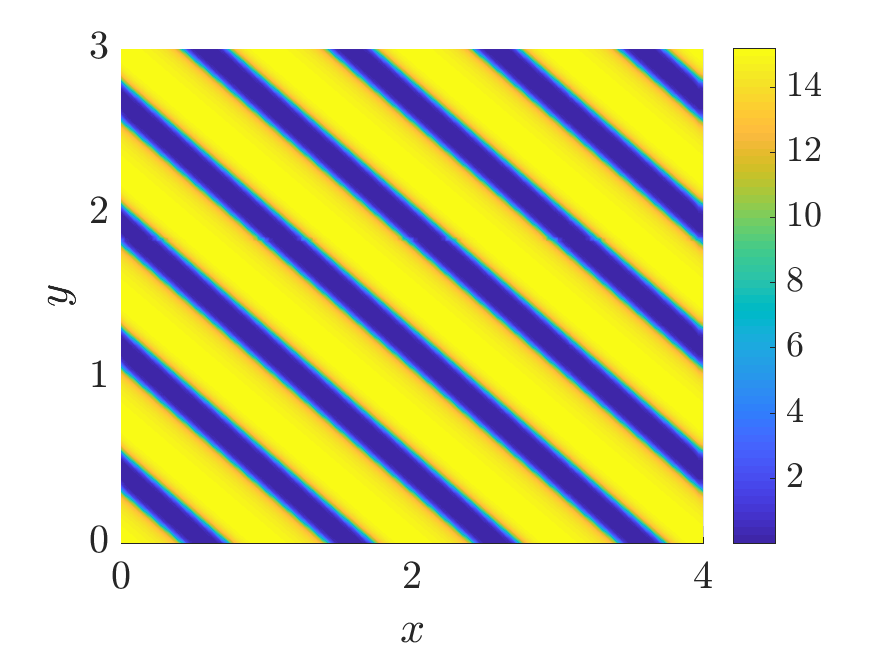}
   \end{minipage}
    \end{tabular}
    \caption{Turing patterns in prey density $u_1$ at final time: (top left) linear diffusion for predators and preys ($d_1 = 0.001$, $d_2 = 4$, $d_{21} = 0$);  (top right) nonlocal cross-diffusion for predators with symmetric kernel and linear diffusion for preys ($d_1 = 0.001$, $d_2 = 0$, $d_{21} = 2/5$, $\rho_2 = \rho_2^\text{sym}$); (bottom) nonlocal cross-diffusion for predators with non-symmetric kernel and linear diffusion for preys ($d_1 = 0.001$, $d_2 = 0$, $d_{21} = 2/5$, $\rho_2 = \rho_2^\text{nonsym}$).}
    \label{fig:turing2d}
\end{figure}

\subsection*{Acknowledgment} MH acknowledges support from the LabEx CEMPI (ANR-11-LABX-0007) and the ministries of Europe and Foreign Affairs (MEAE) and Higher Education, Research and Innovation (MESRI) through PHC Amadeus 46397PA. AZ acknowledges support from the multilateral project of the Austrian Agency
for International Co-operation in Education and Research (OeAD), grant FR 01/2021.

\bibliographystyle{plain}
\bibliography{bibli}

\end{document}